\documentclass{article}

\usepackage{pdfsync}

\usepackage{booktabs}
\usepackage[english]{babel}
\usepackage[utf8x]{inputenc}
\usepackage[a4paper]{geometry}
\usepackage{lmodern}
\usepackage{abstract}
\usepackage{float}
\usepackage{xcolor}
\usepackage{comment}
\usepackage{tikz}
\usepackage{pgfplots}
\usepackage{pgfplotstable}
\usepackage{readarray}
\usepackage{subcaption}

\usepackage{graphicx}
\usepackage[T1]{fontenc}
\usepackage{microtype}
\usepackage{cite}
\usepackage{array,colortbl}
\usepackage{amsmath,amsthm,mathtools}
\usepackage{amsfonts,bm,amssymb}
\usepackage{algorithm}
\usepackage{algorithmic}

\usepackage[hidelinks,hypertexnames=false,hyperindex=true,pdfpagelabels,linktoc=all]{hyperref}
\usepackage{bookmark}
\usepackage{enumitem}

\DeclareMathOperator{\supp}{supp}

\newcommand{\st}{\: : \:}

\makeatletter
\newcommand{\vast}{\bBigg@{4}}
\newcommand{\Vast}{\bBigg@{5}}
\makeatother

\newcommand{\numberthis}{\addtocounter{equation}{1}\tag{\theequation}}

\allowdisplaybreaks

\DeclareFontFamily{U}{mathx}{\hyphenchar\font45}
\DeclareFontShape{U}{mathx}{m}{n}{
<5> <6> <7> <8> <9> <10>
<10.95> <12> <14.4> <17.28> <20.74> <24.88>
mathx10
}{}
\DeclareSymbolFont{mathx}{U}{mathx}{m}{n}
\DeclareFontSubstitution{U}{mathx}{m}{n}
\DeclareMathAccent{\widecheck}{0}{mathx}{"71}
\def\citep#1#2{\cite[{#1}]{#2}}
\newcommand{\bbN}{{\mathbb{N}}}

\newcommand{\bbP}{{\mathbb{P}}}

\newcommand{\bbZ}{{\mathbb{Z}}}
\newcommand{\N}{{\mathbb{N}}} 
\newcommand{\Z}{{\mathbb{Z}}} 

\DeclareSymbolFont{bbold}{U}{bbold}{m}{n}
\DeclareSymbolFontAlphabet{\mathbbold}{bbold}

\newcommand{\bbone}{{\mathbbold{1}}}

\newcommand{\bsh}{{\boldsymbol{h}}}

\newcommand{\bsx}{{\boldsymbol{x}}}

\newcommand{\bsz}{{\boldsymbol{z}}}

\newcommand{\bszero}{{\boldsymbol{0}}} 

\newcommand{\bsgamma}{{\boldsymbol{\gamma}}}

\newcommand{\calH}{{\mathcal{H}}}

\newcommand{\calL}{{\mathcal{L}}}

\newcommand{\calO}{{\mathcal{O}}}
\newcommand{\calP}{{\mathcal{P}}}

\newcommand{\calU}{{\mathcal{U}}}

\newcommand{\calZ}{{\mathcal{Z}}}

\newcommand{\fraku}{{\mathfrak{u}}}

\newcommand{\setu}{{\mathfrak{u}}}

\newcommand{\rme}{{\mathrm{e}}}

\newcommand{\rmi}{{\mathrm{i}}}

\newcommand{\ceil}[1]{\left\lceil #1 \right\rceil}    
\newcommand{\rd}{\,\mathrm{d}} 

\makeatletter 
  \providecommand*{\toclevel@author}{999}
  \providecommand*{\toclevel@title}{0}
\makeatother

\theoremstyle{plain}
  \newtheorem{theorem}{Theorem}[section]
  \newtheorem{proposition}[theorem]{Proposition}
  \newtheorem{lemma}[theorem]{Lemma}
  
  \newtheorem{alg}[theorem]{Algorithm}
\theoremstyle{remark}
  \newtheorem{remark}[theorem]{Remark}

\makeatletter
\newcommand{\doublewidetilde}[1]{{%
  \mathpalette\double@widetilde{#1}%
}}
\newcommand{\double@widetilde}[2]{%
  \sbox\z@{$\m@th#1\widetilde{#2}$}%
  \ht\z@=.9\ht\z@
  \widetilde{\box\z@}%
}
\makeatother

\newcommand{\Pn}{\mathcal{P}_n}
\newcommand{\norm}[1]{\|#1\|_{d,\alpha,\bsgamma}}
\newcommand{\ralpha}{r_{\alpha,\bsgamma}}
\newcommand{\eran}[1]{e_{#1,\alpha,\bsgamma}^{\mathrm{\textup{ran}}}}

\newcommand{\edet}[1]{e_{#1,\alpha,\bsgamma}^{\mathrm{\textup{det}}}}

\newcommand{\spaceH}{\calH_{d,\alpha,\bsgamma}}
\newcommand{\Tquant}[1]{T^{(#1)}}
\newcommand{\thetaquant}{\Theta}

\newcommand{\muquant}{\mu_{d,\alpha,\bsgamma}(\lambda)}

\newcommand{\Zdnull}{\Z^d \setminus \{\bszero\}}

\newcommand{\Euler}[2]{\rme^{{#1} 2 \pi \rmi {#2}}}
\newcommand{\constPn}{c'}

\def\calO{O}

\definecolor{darkred}{RGB}{139,0,0}
\definecolor{darkgreen}{RGB}{0,100,0}
\definecolor{darkmagenta}{RGB}{170,0,120}
\definecolor{darkpurple}{RGB}{110,0,180}
\definecolor{darkblue}{RGB}{40,0,200}
\definecolor{darkbrown}{rgb}{0.75,0.40,0.15}

\newcommand{\randprandzstar}{Q^{\textrm{\textup{ran-pr-vec}}}_{d,n,\tau}}
\newcommand{\randprandzstarsubsec}{Q^{\textrm{\textup{\textnormal{ran-pr-vec}}}}_{d,n,\tau}}

\newcommand{\randprandCBCzstar}{Q^{\textrm{\textup{ran-pr-cbc-vec}}}_{d,n,\tau}}

\newcommand{\randpstarz}{Q^{\textrm{\textup{ran-pr}}}_{d,n,\bsz}}
\newcommand{\randpstarzs}{Q^{\textrm{\textup{ran-pr}}}_{s,n,\bsz}}

\newcommand{\randpstarzstar}{Q^{\textrm{\textup{ran-pr}}}_{d,n,\bsz^\star}}
\newcommand{\randpstarzstars}{Q^{\textrm{\textup{ran-pr}}}_{s,n,\bsz^\star}}
\newcommand{\randpstarzstarPn}{Q^{\textrm{\textup{ran-pr}}}_{d,n,\bsz^\star_{\Pn}}}

\newcommand{\randpstarsminusone}{Q^{\textrm{\textup{ran-pr}}}_{s-1,n,\bsz'}}
\newcommand{\randpstarone}{Q^{\textrm{\textup{ran-pr}}}_{1,n,z_1=1}}

\begin{document}

\title{Random-prime--fixed-vector
randomised lattice-based algorithm for high-dimensional integration}

\author{Frances Y.\ Kuo, Dirk Nuyens, Laurence Wilkes}

\date{April 2023}

\maketitle

\begin{abstract}
We show that a very simple randomised algorithm for numerical integration
can produce a near optimal rate of convergence for integrals of
functions in the $d$-dimensional weighted Korobov space. This
algorithm uses a lattice rule with a fixed generating vector and the only
random element is the choice of the number of function evaluations.
For a given computational budget $n$ of a maximum allowed number of
function evaluations, we uniformly pick a prime $p$ in the range $n/2
< p \le n$. We show error bounds for the randomised error, which is
defined as the worst case expected error, of the form $\calO(n^{-\alpha -
1/2 + \delta})$, with $\delta > 0$, for a Korobov space with smoothness $\alpha > 1/2$ and
general weights. The implied constant in the bound is
dimension-independent given the usual conditions on the weights. We
present an algorithm that can construct suitable generating vectors
\emph{offline} ahead of time at cost $\calO(d n^4 / \ln n)$ when
the weight parameters defining the Korobov spaces are so-called product
weights. For this case, numerical experiments confirm our theory that the
new randomised algorithm achieves the near optimal rate of the randomised
error.
\end{abstract}

\section{Introduction}\label{section:Introduction}

We consider the problem of high-dimensional integration over the unit
cube
$$
  I_d(f)
  :=
  \int_{[0,1]^d} f(\bsx) \rd\bsx
$$
in the \emph{weighted Korobov space} \cite{NW2008,NW2010,SW2001} of one-periodic
functions with smoothness parameter $\alpha>1/2$. If $\alpha$ is an
integer then the space consists of functions whose mixed derivatives up to
order $\alpha$ in each coordinate direction are square-integrable.
Weights were introduced in \cite{SW2001}, see also \cite{NW2008,NW2010,NW2012}, to obtain
tractability results, and they assign an importance to each subset of the dimensions.
Further details on the function space are given in Section~\ref{section:function_space}.
It is known that deterministic cubature rules with $n$ sample points can
achieve the optimal \emph{worst case error} convergence rate of nearly
$\calO(n^{-\alpha})$, while randomised algorithms with a computational
budget of $n$ function evaluations can, for the worst case expected
error, also called the \emph{randomised error}, achieve the optimal
convergence rate of $\calO(n^{-\alpha-1/2})$, which is half an order
better, see, e.g., \cite{U2017,NW2008,B1961}. We formally define both error
criteria in Section~\ref{section:worst_case_errors}.

A deterministic cubature rule is a weighted average of function values of
the integrand $f$, where the cubature points, the cubature weights, as
well as the number of cubature points are pre-determined. In particular, a
deterministic (rank-$1$) \emph{lattice rule}, see, e.g., \cite{SJ1994,N1992,DKP2022}, with $n$ points and a given
generating vector $\bsz \in \mathbb{Z}^d$ takes the form
\begin{align}\label{eq:Q}
  Q_{d,n,\bsz}(f)
  :=
  \frac{1}{n} \sum_{k=0}^{n-1} f\left(\frac{k\bsz \bmod{n}}{n}\right)
  ,
\end{align}
from which it is easily seen that each of the $d$ components of the
generating vector $\bsz = (z_1,\ldots,z_d)$ can be restricted to the set
$\{0,1,\ldots,n-1\}$. Thus, there are $n^d$ possible choices of generating
vectors --- a number that is too large to exhaustively search through in
practice to find the best one. Fortunately, by now it is well-known that
good generating vectors can be constructed \emph{offline} using a
\emph{component-by-component \textnormal{(}CBC\textnormal{)}
construction} such that the lattice rule achieves the optimal worst case error convergence rate
close to $\calO(n^{-\alpha})$, and the implied constant can be independent
of the dimension~$d$ under appropriate conditions on the function space
weights, see~\cite{K2003, DSWW2006}. In a nutshell, the CBC
construction is a greedy strategy which determines the generating vector~$\bsz$
one component at a time, with only~$n$ choices to consider in the
step for selecting the component $z_s$ while holding previously chosen
components $z_1,\ldots,z_{s-1}$ fixed. In the case of product-type
weights, the construction cost is $\calO(d\, n \ln n)$ operations using
the fast Fourier transform, see \cite{NC2006-1,NC2006-2}.

On the other hand, a randomised algorithm is one for which part
of the algorithm is randomly determined \emph{online} at runtime \cite{NW2008,TWW1988,DKP2022}. The
following algorithm was shown by Kritzer, Kuo, Nuyens and Ullrich in 2019
\cite{KKNU2019} to achieve a near optimal randomised error of
$\calO(n^{-\alpha-1/2+\epsilon})$, for~$\epsilon > 0$.
\begin{alg}[``Random-prime--random-vector'' algorithm $\randprandzstar$, see~\eqref{eq:ran-p-z} below]\label{alg:RPRV}
\mbox{}
\begin{enumerate}
\item Select a random prime $p$ from the set $\calP_n$ of primes which satisfy $n/2 < p \leq n$.
\item Select a random generating vector $\bsz \in G^{(p)}$ from the set of good generating vectors for~$p$.
\item Use the lattice rule with generating vector $\bsz$ and $p$ sample points.
\end{enumerate}
\end{alg}

The idea of the ``random-prime--random-vector'' algorithm from
\cite{KKNU2019} originated from Bakhvalov in 1961 \cite{B1961}. The
precise definition of the set of ``good generating vectors'' $G^{(p)}$
involves some extra parameters as subscripts, see
\eqref{eq:good_sets} below. For now it suffices to think of $G^{(p)}$
as the top $50\%$ of the~$p^d$ choices of generating vectors with the
smallest worst case errors. Since the size of the set~$G^{(p)}$ is still
exponentially large in $d$, it is not possible to pre-compute and store
the good set for each $p$.

A constructive algorithm was analysed by Dick, Goda and Suzuki in 2022
\cite{DGS2022} to also achieve a near optimal randomised error.
\begin{alg}[``Random-prime--random-CBC-vector'' algorithm $\randprandCBCzstar$, see~\eqref{eq:ran-p-cbc-z} below]\label{alg:RPRCBCV}
\mbox{}
\begin{enumerate}
\item Select a random prime $p$ from the set $\calP_n$ of primes which satisfy $n/2 < p \leq n$.
\item For $s = 1,\ldots,d$, select a random value for $z_s \in
    \widetilde{G}^{(p)}_{s,\bsz'}$ from the set of good
    components for dimension $s$ and prime $p$, with the previously
    chosen components $\bsz' = (z_1,\ldots,z_{s-1})$ fixed.
\item Use the lattice rule with generating vector $\bsz$ and $p$ sample points.
\end{enumerate}
\end{alg}

Without going into the full details, for now we can think of the set of
``good components'' $\widetilde{G}^{(p)}_{s,\bsz'}$ as the top $50\%$ of
the $p$ possible values for the component $z_s$ with the smallest worst
case errors in dimension~$s$. Since there are only $p$ possible values in
total, it is now computationally feasible to check all of them to find the
better half. Indeed, evaluating the worst case errors for all $p$ values
of $z_s$ can be done using the fast Fourier transform at
$\calO(p\,\ln p)$ operations (assuming product-type weights). We can then
sort the resulting worst case errors and randomly select a value of $z_s$
from the half with the smallest worst case errors. Repeating this for all
$d$ components leads to a total \emph{online} computational cost of
$\calO(d\,n\ln n)$ operations. This
``random-prime--random-CBC-vector'' algorithm from \cite{DGS2022} is a
fully \emph{online} algorithm, since the prime $p$ and all the components
of the generating vector $\bsz$ need to be selected randomly at runtime.

In this paper we propose a new randomised algorithm which also achieves a near optimal randomised error rate. Our new algorithm only selects the prime
$p$ \emph{online} and then makes use of a fixed generating vector~$\bsz^\star$
that is precomputed \emph{offline}. The \emph{online} computational cost
is therefore minimal and we are able to spend the \emph{offline} computation time choosing a vector which minimises the randomised error quantity, for which we will propose an algorithm in Section~\ref{section:construction_of_random_vec}.
\begin{alg}[``Random-prime--fixed-vector'' algorithm $\randpstarzstar$, see \eqref{eq:ran-p} below]\label{alg:RPFV}
\mbox{}
\begin{enumerate}\setcounter{enumi}{-1}
\item {}\textnormal{[}Offline\textnormal{]} Construct $\bsz^\star$ for the set $\calP_n$ of primes which satisfy $n/2 < p \leq n$.
\item Select a random prime $p$ from the set $\calP_n$.
\item Use the lattice rule with the pre-constructed generating vector~$\bsz^\star$ and $p$ sample points.
\end{enumerate}
\end{alg}

We derive an explicit formula for the randomised error for a given
generating vector~$\bsz$ and a computational budget $n$, and then we
prove the existence of such a single generating vector~$\bsz^\star$ which is
universally good for use with all primes $p$ which satisfy $n/2 < p \leq n$. For
this to work, the components of $\bsz^\star$ need to be as large as the product
of all primes between $n/2$ and $n$, denoted by
\[
 N := \prod_{p\in\calP_n} p.
\]
The size of $N$ can be enormous, since the cardinality of $\calP_n$ is of
the order $n/\ln n$. We therefore utilise the Chinese remainder theorem
\[
  \bbZ_N \cong \bigoplus_{p\in\Pn} \bbZ_p,
\]
to break down a vector $\bsz$ into multiple vectors
\[
  \bsz \cong \langle \bsz^{(p_1)}, \bsz^{(p_2)}, \ldots, \bsz^{(p_L)} \rangle, \qquad L := |\calP_n|,
\]
so that the generating vector is only ever considered modulo one of the
primes in $\calP_n$.

Next we propose a CBC construction to obtain this universally good single
generating vector~$\bsz^\star$ for a given computational budget $n$,
\emph{offline}. The evaluation of the randomised error formula for one
generating vector already requires $\calO(d\,n^4/(\ln n)^2)$ operations.
A naive CBC construction would have a cost of $\calO(d\, n^{n+3})$
operations which would be prohibitive. Instead, we propose a construction
which not only operates dimension by dimension, but also prime by prime,
with a computational cost of only $\calO(d\,n^4 / \ln n)$ operations, which is
effectively the same as the cost for evaluating the randomised error for
one generating vector. We stress that this is an \emph{offline} cost and the
vector can be precomputed and stored.

There are other randomised algorithms. The Frolov algorithm \cite{KN2017,U2017}
is also able to obtain the optimal randomised error. It achieves this by scaling an ``irrational lattice'' such that the number of points in the unit cube is a random element.
The ``median of means'' algorithm \cite{GlE2022} uses several randomly generated generating vectors and then takes as the final approximation the median of all corresponding lattice rule approximations to obtain the optimal rate of the deterministic worst case error with high probability.

The structure of the paper is as follows. In
Section~\ref{section:preliminaries_and_setting}, we provide necessary
background of similar results in the deterministic setting and also
look at the recent developments in the randomised setting. In
Section~\ref{section:existence_of_a_single_vector_obtaining_the_optimal_error},
we prove that there exist generating vectors for which Algorithm~\ref{alg:RPFV}
achieves the near optimal rate of randomised error convergence. In
Section~\ref{section:construction_of_random_vec}, we propose and analyse a
method to construct a fixed vector for Algorithm~\ref{alg:RPFV}.
Finally, in Section~\ref{section:numerical_results}, we present some
numerical results for Algorithm~\ref{alg:RPFV} with well constructed
generating vectors.

\section{Preliminaries and settings}\label{section:preliminaries_and_setting}

In this section we reintroduce Algorithm~\ref{alg:RPRV},
$\randprandzstar$, from \cite{KKNU2019} for the case $\alpha > 1/2$ and
provide some necessary background. We extend the proof of
\cite[Theorem~9]{KKNU2019} from product weights to general weights. Note that this reproduction introduces a new method of proof
which will also be applied in
Section~\ref{section:existence_of_a_single_vector_obtaining_the_optimal_error}
for our new algorithm, Algorithm~\ref{alg:RPFV}.

The notation in this paper is as follows. We will consider the ring $\Z_n$ to be
equivalent to~$\{0,\ldots,n-1\}$.
We define the set of natural
numbers to be $\mathbb{N} := \{1, 2, \ldots\}$. We use the shorthand~$\equiv_n$ to mean congruent modulo~$n$. We will use $\bbone(\cdot)$ to
denote the indicator function. When we write $\bbone(\bsh \equiv_p
\bszero)$ for some $\bsh \in \Z^d$ then this is equivalent to $\prod_j
\bbone(h_j \equiv_p 0)$. With $\supp(\bsh)$ we mean the set of indices $j
\in \{1,\ldots,d\}$ for which $h_j \neq 0$.

\subsection{Weighted Korobov spaces}\label{section:function_space}

Let $\bsgamma = \{ \gamma_\setu \}_{\setu \subset \N}$ be a collection of
positive and finite weights. The weighted Korobov space \cite{NW2008,NW2010,SW2001} in $d$ dimensions
with smoothness parameter $\alpha > 1/2$ and weights $\bsgamma$ is defined by
$$
  \spaceH
  :=
  \left\{ f \in \calL_2([0,1]^d)  \st  \norm{f} < \infty \right\}
  ,
$$
with norm
$$
  \norm{f}
  :=
  \Bigg(\sum_{\bsh \in \Z^d} \ralpha^2(\bsh) \, |\widehat{f}(\bsh)|^2 \Bigg)^{1/2}
  ,
$$
where
\begin{align} \label{eq:r-def}
  \ralpha(\bsh)
  :=
  \gamma_{\supp(\bsh)}^{-1} \prod_{j \in \supp(\bsh)}  |h_j|^\alpha
  \qquad\text{and}\qquad
  \widehat{f}(\bsh)
  :=
  \int_{[0,1]^d} f(\bsx) \, \rme^{-2\pi \rmi \bsh\cdot \bsx} \rd \bsx
  ,
\end{align}
with $\bsh \cdot \bsx := \sum_{j=1}^d h_j x_j$.
In this situation the weights control how much each subset $\bsx_\fraku = (x_j)_{j \in \fraku}$ of the variables contributes to the norm of $f$ and, so long as these weights decay sufficiently fast, we can say something for the tractability of the integration problem for functions in these spaces. The weights here are defined for every possible subset of the variables and so we call these \emph{general weights}. Other systems of weights exist such as \emph{product weights}, \emph{POD weights} \cite{KSS2012} and \emph{SPOD weights} \cite{DKlGNS2014}. These forms of weights are specified by $O(d)$ parameters rather than~$O(2^d)$. In particular, product weights are given by $\gamma_\fraku = \prod_{j \in \fraku} \gamma_j$ based on a prescribed sequence $\{\gamma_j\}_{j=1}^\infty$. In Section~\ref{section:complexity_analysis}, we will focus on constructing and analysing our results for product weights.

Note that $\spaceH$ for $\alpha > 1/2$ consists of absolutely converging
Fourier series, since for $f \in \spaceH$ with $\alpha > 1/2$ we have
$$
  \sum_{\bsh \in \Z^d} |\widehat{f}(\bsh)|
  =
  \sum_{\bsh \in \Z^d} |\widehat{f}(\bsh)| \, \frac{\ralpha(\bsh)}{\ralpha(\bsh)}
  \le
  \norm{f} \, \Bigg(\sum_{\bsh \in \Z^d} \ralpha^{-2}(\bsh)\Bigg)^{1/2}
  <
  \infty
  ,
$$
where the sum over $\bsh$ can be rewritten in terms of the Riemann
zeta function $\zeta(2\alpha)$, which is finite provided that $2\alpha >
1$. Indeed, for $\lambda > 0$ we can write
\begin{equation}\label{eq:mu_quantity}
  \muquant
  :=
  \sum_{\bsh \in \Zdnull}  \ralpha^{-1/\lambda}(\bsh)
  =
  \sum_{\emptyset \neq \fraku \subseteq \{1,\ldots,d\}} \gamma_{\fraku}^{1/\lambda}
  \left(2\zeta(\alpha/\lambda) \right)^{|\fraku|},
\end{equation}
which is finite provided that $\alpha/\lambda > 1$, and therefore we
restrict $\lambda$ to the interval $\lambda \in [1/2, \alpha)$. This
quantity $\muquant$ will play an important role in our derivations below.

\begin{remark}
We note that the existence result in
Theorem~\ref{theorem:existence_of_optimal_vector} below will hold in
similarly defined function spaces if we change the definition of
$\ralpha$ in \eqref{eq:r-def} (with the same the definition of $\muquant$
in terms of $\ralpha$) as long as the new definition of $\ralpha$ also
satisfies
    $$
    \ralpha(\bsh) > 0
    \qquad
    \forall \bsh \in \Z^d
    \qquad
    \text{and}
    \qquad
    \muquant < \infty
    \qquad
    \forall \lambda \in [1/2,\alpha)
    ,
    $$
    and
    \begin{align}\label{eq:ralpha-property}
      \ralpha(n \, \bsh)
      &\geq
      n^\alpha \, \ralpha(\bsh)
      \qquad
      \text{whenever $n \in \mathbb{N}$ and $\bsh \in \Zdnull$}
      .
    \end{align}
    In particular, this includes the ``periodic Sobolev spaces of dominating mixed smoothness'' from \cite[Remark~1]{KKNU2019}.
\end{remark}

\subsection{The worst case error and the randomised error}\label{section:worst_case_errors}

Following \cite{NW2008}, for a deterministic algorithm
$A_{d,n}$, we define
its \emph{worst case error} for numerical integration in the space $\spaceH$ as
$$
  \edet{d}(A_{d,n})
  :=
  \sup_{\substack{f\in \spaceH \\ \norm{f} \leq 1}} \left| A_{d,n}(f)-I_d(f) \right|
  ,
$$
and for a randomised algorithm $A^\text{ran}_{d,n}$,
we define its \emph{randomised error}, which is
the \emph{worst case expected error}, as
$$
\eran{d}(A^\text{ran}_{d,n})
:=
\sup_{\substack{f \in \spaceH \\ \norm{f} \leq 1}}
\mathbb{E}\left[ \, | A^\text{ran}_{d,n}(f) - I_d(f) | \, \right]
,
$$
where the expectation is taken with respect to the random elements of
the algorithm.

Note that the randomised error of a deterministic algorithm $A_{d,n}$ corresponds to the usual deterministic worst case error.

\subsection{The deterministic setting}\label{section:deterministic_setting}

The following results are easy extensions of well-known results. Their
proofs are presented in the appendix for completeness. Due to the
absolutely converging Fourier series representation of functions in our
function space, since $\alpha > 1/2$, the error of the lattice
rule approximation~\eqref{eq:Q} can be written as follows \cite{SJ1994}
\begin{equation}\label{eq:lattice_rule_error}
  Q_{d,n,\bsz}(f)-I_d(f)
  =
  \sum_{\substack{\bsh \in \Zdnull \\ \bsh\cdot\bsz \equiv_n 0}} \widehat{f}(\bsh)
  .
\end{equation}
This then leads to the following explicit form of the worst case error.

\begin{proposition}\label{prop:edet}
    For $\alpha > 1/2$ we have
    \begin{equation*}
        \edet{d}(Q_{d,n,\bsz})
        =
        \Bigg(\sum_{\substack{\bsh\in \Zdnull \\ \bsh\cdot\bsz \equiv_n 0}} \ralpha^{-2}(\bsh)\Bigg)^{1/2}
        .
    \end{equation*}
\end{proposition}
\begin{proof}
  See \hyperlink{proof:prop:edet}{appendix}.
\end{proof}

In Proposition~\ref{proposition:deterministic_bound_is_achieved}
below we will present the well-known result that there exist
generating vectors
which can achieve the optimal rate of convergence for the worst case error.

In order to prove this and later statements, the following lemma
is required. We remind the reader that we write $\bsh \equiv_p
\bszero$ to denote that every component of $\bsh$ is a multiple of~$p$,
i.e., $\bbone(\bsh \equiv_p \bszero) = \prod_j \bbone(h_j \equiv_p 0)$.

\begin{lemma}\label{lemma:averaging_over_vecs}
    For
    $\bsh \in \Z^d$ and prime $p$,
    $$
        \frac{1}{p^d}\sum_{\bsz \in \Z_p^d} \bbone(\bsh\cdot\bsz \equiv_p 0)
        =
        \frac{p-1}{p} \, \bbone(\bsh \equiv_p \bszero) + \frac{1}{p}
        .
    $$
    Furthermore, for any subset $\calZ \subseteq \Z_p^d$ with $|\calZ| \geq \ceil{\tau p^d}$ and $0 < \tau \le 1$, we have
    $$
    \frac{1}{|\calZ|}\sum_{\bsz \in \calZ} \bbone(\bsh\cdot\bsz \equiv_p 0)
    \leq
    \frac{\tau p - 1}{\tau p} \, \bbone(\bsh \equiv_p \bszero) + \frac{1}{\tau p}
    .
    $$
\end{lemma}
\begin{proof}
  See \hyperlink{proof:lemma:averaging_over_vecs}{appendix}.
\end{proof}

\begin{proposition}\label{proposition:deterministic_bound_is_achieved}
    For $\alpha > 1/2$, $\lambda \in [1/2,\alpha)$ and prime $p$, we have
    \begin{equation*}
        \frac{1}{p^d}\sum_{\bsz \in \Z_p^d} \left[\edet{d}(Q_{d,p,\bsz})\right]^{1/\lambda}
        \leq
        \frac{2}{p} \, \muquant
        ,
    \end{equation*}
where $\muquant$ is defined in \eqref{eq:mu_quantity}.
    Thus at least one vector $\bsz^\star \in \Z_p^d$ satisfies the bound
    \begin{equation}\label{eq:edet-at-least-one}
      \edet{d}(Q_{d,p,\bsz^\star})
      \leq
      \left( \frac{2}{p} \, \muquant \right)^{\lambda}
      \qquad \forall \lambda \in [1/2,\alpha)
      .
    \end{equation}
In particular, the vector $\bsz^\star$ which minimises
$\edet{d}(Q_{d,p,\bsz})$ over all $\bsz \in \Z_p^d $ achieves this bound.
\end{proposition}
\begin{proof}
  See \hyperlink{proof:proposition:deterministic_bound_is_achieved}{appendix}.
\end{proof}

It is important to realise that $\bsz^\star$ in the above statement can be chosen as the minimiser of~$\edet{d}(Q_{d,p,\bsz})$, which is independent of $\lambda$, and hence the bound~\eqref{eq:edet-at-least-one} holds for all $\lambda \in [1/2, \alpha)$.
Unfortunately the set $\Z_p^d$ is of size $p^d$ which is way too big to test them all.

 Therefore we turn to the component-by-component (CBC) construction,
see, e.g., \cite{K2003}, for which there are fast algorithms that can
find a good vector in e.g., $\calO(d \, n \ln n)$ operations in the case
of product weights, see \cite{NC2006-1,NC2006-2}. The CBC
construction argument relies on an induction over the dimension
$s = 2,\ldots,d$. If the first $s-1$ components have already
been constructed, $\bsz' = (z_1,\cdots,z_{s-1})$, and we have a good error
bound for $\edet{s-1}(Q_{s-1,p,\bsz'})$, then it is shown that the $s$th
component can be found such that $\edet{s}(Q_{s,p,\bsz})$ has a good error
bound. For this we split the squared worst case error into two terms
depending on whether or not $h_s = 0$:
\begin{align}\label{eq:det-cbc}
 \left[\edet{s}(Q_{s,p,\bsz})\right]^2
 = \sum_{\substack{\bsh\in\bbZ^s\setminus\{\bszero\} \\ \bsh\cdot\bsz\equiv_p 0}}
  \ralpha^{-2}(\bsh)
 = \underbrace{\sum_{\substack{\bsh'\in\bbZ^{s-1}\setminus\{\bszero\} \\ \bsh'\cdot\bsz'\equiv_p 0}}
   \ralpha^{-2}(\bsh')}_{\left[\edet{s-1}(Q_{s-1,p,\bsz'})\right]^2}
  \;+\; \underbrace{
   \sum_{\substack{\bsh\in\bbZ^s \\ h_s \ne 0 \\ \bsh\cdot\bsz\equiv_p 0}}
   \ralpha^{-2}(\bsh)}_{=:\, \theta_s^{(p)}(z_s)}
   .
\end{align}
The first term is precisely the squared worst case error in dimension
$s-1$ and it has the correct order of convergence from the induction
hypothesis. What remains is to show that the second term
\begin{align*}
 \theta^{(p)}_s(z_s)
 =
 \theta^{(p)}_{s,\alpha,\bsgamma,\bsz'}(z_s)
 &:=
 \sum_{\substack{\bsh\in\bbZ^s \\ h_s \ne 0 \\ \bsh\cdot\bsz\equiv_p 0}}
  \ralpha^{-2}(\bsh)
\end{align*}
has the correct order.
For completeness we show the equivalent result to
Proposition~\ref{proposition:deterministic_bound_is_achieved}.

\begin{proposition}\label{proposition:bound_on_theta}
For $\alpha > 1/2$, $\lambda \in [1/2,\alpha)$ and prime $p$, we have
$$
\frac{1}{p} \sum_{z_s \in \Z_p} \left[\theta_s^{(p)}(z_s)\right]^{1/(2\lambda)}
\leq
\frac{2}{p} \sum_{\substack{\bsh\in\bbZ^{s} \\ h_s \ne 0}} \ralpha^{-1/\lambda}(\bsh)
.
$$
Thus at least one $z_s^\star\in\bbZ_p$ satisfies
\begin{align*}
\theta_s^{(p)}(z_s^\star)
  \le
\bigg( \frac{2}{p} \sum_{\substack{\bsh \in\bbZ^s \\ h_s\ne 0}}
\ralpha^{-1/\lambda}(\bsh) \bigg)^{2\lambda}
\qquad\forall\, \lambda \in [1/2,\alpha).
\end{align*}
In particular, the $z_s^\star$ which minimises $\theta_s^{(p)}(z_s)$
over all $z_s\in\bbZ_p$ achieves this bound.
\end{proposition}
\begin{proof}
  See \hyperlink{proof:proposition:bound_on_theta}{appendix}.
\end{proof}

\subsection{The randomised setting}\label{section:randomised_algorithms}

For the randomised algorithms, we define the set of primes from which the number of sample points will be chosen to be
\begin{align*}
  \Pn
  &:=
  \left\{ \frac{n}{2} < p \le n : p \text{ prime} \right\}
  .
\end{align*}
For $n \ge 2$, we can bound the size of this set from below:
\begin{equation}\label{eq:PNT_lowerbound}
|\Pn|
>
\frac{\constPn \, n}{\ln n}
\quad \text{for some } 0.23 < \constPn < 0.5
,
\end{equation}
where $\constPn$ converges to $0.5$ as $n$ increases, see
\cite{lvP1896,lvP1899}. The bound $0.23$ is verified with results from \cite{D1998} and $\constPn$ is greater than $0.4$ for all $n>40$.

We now formally define the set of ``good generating vectors'' and the
set of ``good components'' as used in
Algorithms~\ref{alg:RPRV}--\ref{alg:RPFV} in the Introduction.

For a given prime $p$ and a parameter $\tau \in (0, 1)$, we
define the set of ``good generating vectors'' (for a lattice rule with $p$
points) to be
\begin{equation}\label{eq:good_sets}
  G_\tau^{(p)}
  =
  G_{d,\alpha,\bsgamma,\tau}^{(p)}
  :=
  \left\{ \bsz \in \Z_p^d  \st \edet{d}(Q_{d,p,\bsz}) \leq \inf_{\lambda \in [1/2,\alpha)}\left[\left(\frac{2}{(1-\tau) \, p} \, \muquant \right)^{\lambda}\right] \right\}
  ,
\end{equation}
where $\muquant$ is defined in \eqref{eq:mu_quantity}.
As we will see later in Lemma~\ref{lemma:good_sets_are_large}, the
size of this set increases with the parameter~$\tau$.

Algorithm~\ref{alg:RPRV} ``random-prime--random-vector'' can
hence be written as
\begin{align}\label{eq:ran-p-z}
  \randprandzstar(f,\omega)
  :=
  Q_{d,p(\omega),\bsz(\omega)}(f)
  \qquad&\text{with }
  p \sim \calU(\Pn) \text{ and then } \bsz \sim \calU\big( G_\tau^{(p)} \big),
\end{align}
where $p \sim \calU(\Pn)$ means that $p$ is uniformly at random selected from $\Pn$ and, likewise, $\bsz$ is uniformly at random selected from $G_\tau^{(p)}$ for the given~$p$.

For a given prime $p$ and a parameter $\tau \in (0,1)$, and for each
$s = 1,\ldots,d$ and with given $\bsz' \in \Z^{s-1}$, we define the
set of ``good components'' (for dimension $s$ of a lattice rule with $p$
points) to be
\begin{align}\label{eq:Gtilde}
  \nonumber
  \widetilde{G}_{s,\tau}^{(p)}
 = \widetilde{G}_{s,\tau,\bsz'}^{(p)}
 &\hphantom{:}= \widetilde{G}_{s,\alpha,\bsgamma,\tau,\bsz'}^{(p)}
 \\
 &:= \Biggr\{
 z_s\in\bbZ_p \st
 \theta^{(p)}_{s,\alpha,\bsgamma,\bsz'}(z_s)
 \le \inf_{\lambda \in [1/2,\alpha)}
 \bigg( \frac{2}{(1-\tau) \, p} \sum_{\substack{\bsh \in\bbZ^d \\ h_s \ne 0}}
 \ralpha^{-1/\lambda}(\bsh) \bigg)^{2\lambda}
 \Biggr\}
 .
\end{align}
The sets $\widetilde{G}_{s,\tau}^{(p)}$ are built to facilitate a CBC
construction; the sets depend on the previously constructed components
$\bsz' = (z_1,\ldots,z_{s-1})$ and consist of the good options for the
component $z_s$ of dimension $s$.

Algorithm~\ref{alg:RPRCBCV} ``random-prime--random-CBC-vector''
can hence be written as
\begin{align}\label{eq:ran-p-cbc-z}
  \randprandCBCzstar(f,\omega)
  :=
  Q_{d,p(\omega),\bsz(\omega)}(f)
  \qquad&\text{with }
  p \sim \calU(\Pn), \; z_1 = 1, \\\nonumber
  & \text{and then } z_2 \sim \calU\big( \widetilde{G}_{2,\tau,(z_1)}^{(p)} \big), \\\nonumber
  & \;\;\,\vdots \\\nonumber
  & \text{and then } z_d \sim \calU\big( \widetilde{G}_{d,\tau,(z_1,z_2,\ldots,z_{d-1})}^{(p)} \big)
  .
\end{align}
Finally, our new Algorithm~\ref{alg:RPFV}
``random-prime--fixed-vector'' is simply given by
\begin{align}\label{eq:ran-p}
  \randpstarzstar(f,\omega)
  :=
  Q_{d,p(\omega),\bsz^\star}(f)
  \qquad&\text{with }
  p \sim \calU(\Pn)
  .
\end{align}

We note that our definitions of Algorithms~\ref{alg:RPRV}
and~\ref{alg:RPRCBCV} here
are slightly different from the respective publications \cite{KKNU2019}
and \cite{DGS2022}. In \cite{KKNU2019} the set of good generating
vectors was defined in terms of the Zaremba index such that it can also
be used for $0 < \alpha \le 1/2$. However, we here choose to define the
set in terms of the worst case error as this is a computable criterion for
a CBC algorithm. We also allow the components of the generating
vector for the $p$-point lattice rule to be any element of $\Z_p$
instead of restricting them to be relatively prime to $p$. This has no
significant effect on the error bounds. Similarly, the range of $\Pn$
is slightly differently defined, but this is inconsequential. In
\cite{DGS2022} the set of good components was defined as the first
$\left\lceil \tau p \right\rceil$ components when ordered in terms of
increasing $\theta_s^{(p)}$ value. Here we define the set of good
components as those $z_s^{(p)}$ that deliver sufficiently small
$\theta_s^{(p)}$ value; this definition is more convenient for error
analysis. The definition in \cite{DGS2022} is advantageous for the actual
implementation and we will also make use of such a point of view for
Algorithm~\ref{alg:constructing_z}.

\subsection{New proof for the error bound on $\randprandzstarsubsec$}

In this subsection we provide a new proof for the error bound
of Algorithm~\ref{alg:RPRV} ``random-prime--random-vector'',
$\randprandzstar$, see~\eqref{eq:ran-p-z}. We extend the result
in~\cite{KKNU2019} from product weights to general weights.

In the definition of $G_\tau^{(p)} = G_{d,\alpha,\bsgamma,\tau}^{(p)}$, see~\eqref{eq:good_sets}, the parameter $\tau$ allows us to relax the bound~\eqref{eq:edet-at-least-one} achieved by Proposition~\ref{proposition:deterministic_bound_is_achieved} and therefore ensures the set can be made sufficiently large as shown in the following result.

\begin{lemma}\label{lemma:good_sets_are_large}
    For prime $p$ and $\tau\in (0,1)$, the cardinality of the set $G_\tau^{(p)}$ defined in \eqref{eq:good_sets} satisfies
    $$
    |G_\tau^{(p)}|
    \geq
    \ceil{\tau p^d}
    \ge 1
    .
    $$
\end{lemma}
\begin{proof}
  See \hyperlink{proof:lemma:good_sets_are_large}{appendix}.
\end{proof}

\begin{lemma}\label{lemma:lower_bound_on_r}
    For prime $p$, $\tau\in (0,1)$, $\bsz \in G_\tau^{(p)}$ and $\bsh \in \Z_d \setminus \{\bszero\}$ satisfying $\bsh \cdot \bsz \equiv_p 0$, we have
    $$
      \ralpha(\bsh)
      \geq
       \sup_{\lambda \in [1/2,\alpha)}
       p^{\lambda} \left[\frac{2}{1-\tau} \, \muquant \right]^{-\lambda}
       ,
    $$
where $\muquant$ is defined in \eqref{eq:mu_quantity}. Furthermore,
if
    $p \in \Pn$ then
    $$
      \ralpha(\bsh)
      >
      B_{n,\tau}
    $$
    where
    \begin{align}\label{eq:B_n}
      B_{n,\tau}
      =
      B_{d,\alpha,\bsgamma,n,\tau}
      :=
      \sup_{\lambda \in [1/2,\alpha)} n^{\lambda} \left[\frac{4}{1-\tau} \, \muquant \right]^{-\lambda}
      .
    \end{align}
\end{lemma}
\begin{proof}
  See \hyperlink{proof:lemma:lower_bound_on_r}{appendix}.
\end{proof}

In the following proposition we provide an expression for the randomised
error of $\randprandzstar$. An upper bound was already given in
\cite{KKNU2019} for product weights and we extend it to general
weights.

The condition $\ralpha(\bsh) > B_{n,\tau}$ under the sum
in~\eqref{eq:eran-M} does not need to be stated explicitly, since it
follows automatically from Lemma~\textup{\ref{lemma:lower_bound_on_r}} and
the definition of~$\omega_{n,\tau}$ in \eqref{eq:omega}, but it will be of
direct use in proving Theorem~\textup{\ref{thm:M-error-bound}}
later.

\begin{proposition}\label{proposition:M_error}
    For $\alpha > 1/2$ and $n \geq 2$ we have
    \begin{align}\label{eq:eran-M}
    \eran{d}(\randprandzstar)
    &=
    \Bigg(\sum_{\substack{\bsh \in \Zdnull \\ \ralpha(\bsh) > B_{n,\tau}}}\omega_{n,\tau}^2(\bsh) \, \ralpha^{-2}(\bsh)\Bigg)^{1/2},
    \end{align}
    with
    \begin{align} \label{eq:omega}
    \omega_{n,\tau}(\bsh)
    =
    \omega_{d,\alpha,\bsgamma,n,\tau}(\bsh)
    :=
    \frac{1}{|\Pn|}\sum_{p \in \Pn} \frac{1}{|G_\tau^{(p)}|}\sum_{\bsz \in G_\tau^{(p)}} \bbone(\bsh\cdot\bsz \equiv_p 0)
    ,
    \end{align}%
    where $G_\tau^{(p)} = G_{d,\alpha,\bsgamma,\tau}^{(p)}$ is defined
    in \eqref{eq:good_sets}.
\end{proposition}
\begin{proof}
    For an arbitrary function $f \in \spaceH$, we have from
    \eqref{eq:lattice_rule_error}
    \begin{align}
        \nonumber
        &\frac{1}{|\Pn|} \sum_{p \in \Pn}
        \frac{1}{|G_\tau^{(p)}|} \sum_{\bsz \in G_\tau^{(p)}}
        |Q_{d,p,\bsz}(f)-I_d(f)|
        =
        \frac{1}{|\Pn|} \sum_{p \in \Pn}
        \frac{1}{|G_\tau^{(p)}|} \sum_{\bsz \in G_\tau^{(p)}}
        \bigg|
        \sum_{\substack{\bsh \in \Zdnull \\ \bsh \cdot \bsz \equiv_p 0}}
        \widehat{f}(\bsh)
        \bigg|
        \\\nonumber
        &\hspace{10mm}\leq
        \frac{1}{|\Pn|} \sum_{p \in \Pn}
        \frac{1}{|G_\tau^{(p)}|}
        \sum_{\bsz \in G_\tau^{(p)}}
        \sum_{\substack{\bsh \in \Zdnull \\ \bsh \cdot \bsz \equiv_p 0}}
        |\widehat{f}(\bsh)|
        =
        \sum_{\bsh \in \Zdnull}
        \omega_{n,\tau}(\bsh) \, |\widehat{f}(\bsh)|
        \\\label{eq:Qranpzub}
        &\hspace{10mm}\leq
        \bigg(
          \sum_{\bsh \in \Zdnull}
          \omega_{n,\tau}^2(\bsh) \, \ralpha^{-2}(\bsh)
        \bigg)^{1/2}
        \norm{f}
        .
    \end{align}
It follows that $\eran{d}(\randprandzstar)$ is bounded from above by
the expression multiplying $\norm{f}$ in~\eqref{eq:Qranpzub}. Since
equality in \eqref{eq:Qranpzub} is attained by the function
    $$
      f(\bsx) = \sum_{\bsh \in \Zdnull}
      \omega_{n,\tau}(\bsh) \, \ralpha^{-2}(\bsh) \, \rme^{2 \pi \rmi \bsh \cdot \bsx}
      \qquad
      \in \qquad \spaceH
      ,
    $$
we then conclude that $\eran{d}(\randprandzstar)$ is exactly the
expression multiplying $\norm{f}$ in \eqref{eq:Qranpzub}.
The condition $\ralpha(\bsh) > B_{n,\tau}$ under the sum
in~\eqref{eq:eran-M} follows automatically from
Lemma~\textup{\ref{lemma:lower_bound_on_r}} and the definition
of~$\omega_{n,\tau}$ in \eqref{eq:omega}.
\end{proof}

The following lemma gives an upper bound on $\omega_{n,\tau}(\bsh)$
in terms of $\ralpha(\bsh)$. The proof technique is similar to \cite[proof
of Theorem~9]{KKNU2019} but we allow general weights and do not require
them to be bounded by~$1$. The proof does not explicitly use the
definition of the $G_\tau^{(p)}$ sets, and it is possible to allow $\tau =
1$ with the interpretation that in this case $G_\tau^{(p)} = \Z_p^d$ in
the definition of~$\omega_{n,\tau}$. Moreover, the result also holds for
any $\alpha > 0$.

\begin{lemma}\label{lemma:omega_bound}
Suppose the weights satisfy $\gamma_\setu \le \Gamma < \infty$ for all
$\setu\subset\bbN$. For $\alpha
> 0$, $n \ge 4$, $\tau \in (0, 1]$, and $\bsh \in \Z^d \setminus
\{\bszero\}$, the quantity $\omega_{n,\tau}(\bsh)$ defined in
\eqref{eq:omega} satisfies
  \begin{align} \label{eq:omega_bound}
    \omega_{n,\tau}(\bsh)
    =
    \omega_{d,\alpha,\bsgamma,n,\tau}(\bsh)
    \le
    C_{\tau,\delta} \,
    \frac{\ralpha^{\delta/\lambda}(\bsh)}{n}
    \qquad
    \text{for all } 0 < \delta \le \lambda \le \alpha
  \end{align}
  with $C_{\tau,\delta} = \left( \frac{2}\tau + \frac{2}{\constPn \delta} \right)
  \Gamma$, where $c'$ is as in \eqref{eq:PNT_lowerbound}.
\end{lemma}
\begin{proof}
  See \hyperlink{proof:lemma:omega_bound}{appendix}.
\end{proof}

The following theorem extends~\cite[Theorem~9]{KKNU2019} to
general weights. We use a more compact proof technique which will be
applied in later results.

\begin{theorem}\label{thm:M-error-bound}
    For $d \in \N$, $\alpha>1/2$, positive weights $\bsgamma = \{ \gamma_\setu \}_{\setu \subset \N}$,
    $\tau \in (0,1)$ and $n \ge 4$,
    the randomised error for integration of functions from the space $\spaceH$ using the randomised lattice algorithm
    $\randprandzstar$ in \eqref{eq:ran-p-z}
    satisfies
    $$
    \eran{d}(\randprandzstar)
    \leq
    \frac{C_{\tau,\delta,\lambda}}{n^{\lambda+1/2 - \delta}} \, \Big[ \muquant \Big]^{\lambda-\delta}
    \qquad
    \forall \lambda \in (1/2,\alpha)
    \quad
    \forall \delta \in (0, \lambda-1/2]
    ,
    $$
    where $C_{\tau,\delta,\lambda} > 0$ is a constant that goes to infinity for $\tau\to 0$ or $\tau\to 1$ or $\delta\to 0$.
    The quantity $\muquant$, as defined in~\eqref{eq:mu_quantity}, goes to infinity for $\lambda\to\alpha$.
    The upper bound is independent of the dimension if the weights $\bsgamma$ are such that
    $$
    \sum_{\emptyset \neq \fraku \subset \N} \gamma_{\fraku}^{1/\lambda} \left(2\zeta(\alpha/\lambda) \right)^{|\fraku|}
    <
    \infty
    .
    $$
\end{theorem}
\begin{proof}
    We will now combine \eqref{eq:B_n}, \eqref{eq:eran-M} and \eqref{eq:omega_bound}.
    We need to keep track of the conditions on $0 < \delta \le \lambda < \alpha$ when we manipulate the infinite sums below such that they are convergent in each step. This can be done by rewriting them in terms of the Riemann zeta function as was done in~\eqref{eq:mu_quantity}.
    We discuss these conditions after the derivation.
    We obtain
    \begin{align*}
        \eran{d}(\randprandzstar)
        &\quad\leq
        \frac{C_{\tau,\delta}}{n} \Bigg(\sum_{\substack{\bsh \in \Zdnull \\ \ralpha(\bsh) > B_{n,\tau}}} \ralpha^{2\delta/\lambda}(\bsh) \, \ralpha^{-2}(\bsh) \Bigg)^{1/2}
        \\
        &\quad\leq
        \frac{C_{\tau,\delta}}{n}
        \Bigg(
          \sum_{\substack{\bsh \in \Zdnull \\ \ralpha(\bsh) > B_{n,\tau}}}
            \ralpha^{2\delta/\lambda-2}(\bsh)
            \left(\frac{\ralpha(\bsh)}{B_{n,\tau}}\right)^{2-(1+2\delta)/\lambda}
        \Bigg)^{1/2}
        \\[2mm]
        &\quad\leq
        \frac{C_{\tau,\delta}}{n \, B_{n,\tau}^{1-(1+2\delta)/(2 \lambda)}} \Bigg(\sum_{\bsh \in \Zdnull} \ralpha^{-1/\lambda}(\bsh) \Bigg)^{1/2}
        \\
        &\quad\le
        \frac{C_{\tau,\delta}}{n^{\lambda+1/2-\delta}} \Big[\frac{4}{1-\tau} \, \muquant \Big]^{\lambda-1/2-\delta} \Big[ \muquant \Big]^{1/2}
        \\
        &\quad=
        \frac{C_{\tau,\delta,\lambda}}{n^{\lambda+1/2-\delta}} \Big[ \muquant \Big]^{\lambda-\delta},
        &&
    \end{align*}
    with $C_{\tau,\delta,\lambda} := C_{\tau,\delta} \, (4/(1-\tau))^{\lambda-1/2-\delta}$.
For the first inequality we need $(2 - 2\delta/\lambda)\alpha > 1$,
i.e., $\delta < \lambda - \lambda/(2\alpha)$. For the second inequality we
need $2 - (1+2\delta)/\lambda \ge 0$, i.e., $\delta \le \lambda - 1/2$.
Coincidentally, $\delta \le \lambda - 1/2$ implies $\delta < \lambda -
\lambda/(2\alpha)$ since $\lambda < \alpha$. The condition $\delta \le
\lambda - 1/2$ together with $\delta> 0$ means that necessarily $\lambda
>1/2$. In the fourth inequality we applied the bound \eqref{eq:B_n} with the same
$\lambda$ parameter for convenience. Hence the bound holds for all
$\lambda \in (1/2,\alpha)$ and all $\delta\in (0, \lambda-1/2]$.
\end{proof}

\section{Existence of a vector achieving the near optimal error}\label{section:existence_of_a_single_vector_obtaining_the_optimal_error}

In this section we show that there exists a single generating vector $\bsz
\in \Z^d$ for our new Algorithm~\ref{alg:RPFV}
``random-prime--fixed-vector'' $\randpstarz$,
see~\eqref{eq:ran-p}, such that the randomised error
$\eran{d}(\randpstarz)$ has the near optimal
rate of convergence of $\calO(n^{-\alpha-1/2+\delta})$ for arbitrarily
small $\delta > 0$. The components of $\bsz$ will be numbers from $\Z_N$
with $N = \prod_{p\in\Pn} p$ and thus we can equivalently say $\bsz \in
\Z_N^d$. Note that $N$ grows extremely quickly with~$n$, i.e., $|N| \ge
n^{\constPn n / \ln n}$, where $c'$ is as in
\eqref{eq:PNT_lowerbound}.

Let us enumerate the primes in $\Pn$ by $p_1$, \ldots, $p_L$ with $L =
|\Pn|$. Since $\Pn$ is a collection of prime numbers, and hence they are
all relatively prime to each other, we can use the Chinese remainder
theorem to have a one-to-one mapping from the vectors $\bsz \in \Z_N^d$ to
the $L$-tuples of vectors $\langle \bsz^{(p_1)}, \ldots, \bsz^{(p_L)}
\rangle \in \Z_{p_1}^d \oplus \cdots \oplus \Z_{p_L}^d$ satisfying
$\bsz^{(p)} \equiv_p \bsz$ for each prime $p \in \Pn$. The single
generating vector we seek
can hence be equivalently interpreted as $L$ generating vectors, one for
each of the possible lattice rule $Q_{d,p,\bsz^{(p)}} = Q_{d,p,\bsz}$ as a
realisation of at the algorithm $\randpstarz$ can apply. We have
$Q_{d,p,\bsz^{(p)}} = Q_{d,p,\bsz}$ since the generating vector is only
considered modulo~$p$, see~\eqref{eq:Q}.

We first derive an expression for the randomised error of
$\randpstarz$ for any given vector $\bsz \in \Z^d$.

\begin{proposition}\label{prop:K_error}
    For $\alpha > 1/2$ and $n \geq 2$ we have
    \begin{equation*}
        \eran{d}(\randpstarz)
        =
        \bigg(\sum_{\bsh \in \Zdnull} \omega_n^2(\bsh,\bsz) \, \ralpha^{-2}(\bsh)\bigg)^{1/2}
        ,
    \end{equation*}
    where
    \begin{equation}\label{eq:omega-n-z}
      \omega_n(\bsh,\bsz)
      :=
      \frac{1}{|\Pn|} \sum_{p \in \Pn} \bbone(\bsh \cdot \bsz \equiv_p 0)
      .
    \end{equation}
\end{proposition}
\begin{proof}
    For an arbitrary function $f \in \spaceH$ and a given $\bsz$, we have, from \eqref{eq:lattice_rule_error},
    \begin{align}
        \nonumber
        &\frac{1}{|\Pn|}\sum_{p\in \Pn} |Q_{d,p,\bsz}(f)-I_d(f)|
        =
        \frac{1}{|\Pn|}\sum_{p\in \Pn}\bigg|\sum_{\substack{\bsh \in \Zdnull \\ \bsh \cdot \bsz \equiv_p 0}}\widehat{f}(\bsh)\bigg|
        \\\nonumber
        &\hspace{10mm}\leq
        \frac{1}{|\Pn|}\sum_{p\in \Pn}\sum_{\substack{\bsh \in \Zdnull \\ \bsh \cdot \bsz \equiv_p 0}} \, |\widehat{f}(\bsh)|
        =
        \sum_{\bsh \in \Zdnull} \omega_n(\bsh,\bsz) \, |\widehat{f}(\bsh)|
        \\\label{eq:Qranpub}
        &\hspace{10mm}=
        \bigg(\sum_{\bsh \in \Zdnull} \omega_n^2(\bsh,\bsz) \, \ralpha^{-2}(\bsh)\bigg)^{1/2} \, \norm{f}
        .
    \end{align}
    We see that the function
    $$
    f(\bsx)
    =
    \sum_{\bsh \in \Zdnull} \omega_n(\bsh, \bsz) \, \ralpha^{-2}(\bsh) \, \rme^{2 \pi \rmi \bsh \cdot \bsx}
    \qquad
    \in \qquad \spaceH
    $$
    achieves equality for~\eqref{eq:Qranpub} and so we obtain an explicit expression for the worst case error.
\end{proof}

\begin{theorem}\label{theorem:existence_of_optimal_vector}
    For $d \in \N$, $\alpha>1/2$, positive weights $\bsgamma = \{ \gamma_\setu \}_{\setu \subset \N}$ and $n \ge 2$,
    there exists a vector $\bsz^\star \in \Z_{N}^d$ with $N = \prod_{p \in \Pn} p$ such that the randomised error for integration of functions from the space $\spaceH$ using the randomised lattice algorithm $\randpstarzstar$ in \eqref{eq:ran-p}
    satisfies
    $$
    e_{d,\alpha,\bsgamma}^\text{ran}(\randpstarzstar)
    \leq
        \frac{(C_\lambda\,\ln n)^{1/2}}{n^{\lambda+1/2}} \,
        \Big[\muquant\Big]^\lambda
    \qquad
    \forall \lambda \in [1/2,\alpha)
    ,
    $$
    where $C_\lambda > 0$.
    The quantity $\muquant$, as defined in~\eqref{eq:mu_quantity}, goes to infinity for $\lambda \to \alpha$.
    The upper bound is independent of the dimension if the weights $\bsgamma$ are such that
    $$
    \sum_{\emptyset \neq \fraku \subset \N} \gamma_{\fraku}^{1/\lambda} \left(2\zeta(\alpha/\lambda) \right)^{|\fraku|}
    <
    \infty
    .
    $$
\end{theorem}
\begin{proof}
    We make use of the Chinese remainder theorem as explained at the beginning of this section.
    We will show that the average for $\bsz \cong \langle \bsz^{(p_1)}, \ldots, \bsz^{(p_L)} \rangle \in G^{(p_1)}_\tau \oplus \cdots \oplus G^{(p_L)}_\tau$, where $G^{(p)}_\tau = G^{(p)}_{d,\alpha,\bsgamma,\tau}$ are the ``good sets'' as defined in~\eqref{eq:good_sets}, with $\tau \in (0,1)$,
    will be good and hence there must exist a vector which achieves it.
    We will denote the set of choices for the single generating vector $\bsz \in \bbZ_N^d$ by $G_{n,\tau} = G_{d,\alpha,\bsgamma,n,\tau} :\cong G^{(p_1)}_\tau \oplus \cdots \oplus G^{(p_L)}_\tau$.
    The proof holds for any choice of~$\tau \in (0,1)$. We will fix the choice of $\tau$ at the end based on an upper bound to attain the final error bound independent of~$\tau$.

    From Proposition~\ref{prop:K_error} and Lemma~\ref{lemma:lower_bound_on_r} it follows that for $\tau \in (0,1)$
    \begin{align*}
      \forall \bsz \in G_{n,\tau}: \quad
      \left[\eran{d}(\randpstarz)\right]^2
      &=
      \sum_{\substack{\bsh \in \Zdnull \\ \ralpha(\bsh) > B_{n,\tau}}} \omega_n^2(\bsh,\bsz) \, \ralpha^{-2}(\bsh)
      \\
      &\le
      \sum_{\substack{\bsh \in \Zdnull \\ \ralpha(\bsh) > B_{n,\tau}}} \omega_n^2(\bsh,\bsz) \, \ralpha^{-2}(\bsh) \left(\frac{\ralpha(\bsh)}{B_{n,\tau}}\right)^{2-1/\lambda}
      \\
      &\leq
      \frac{1}{B_{n,\tau}^{2-1/\lambda}} \sum_{\bsh \in \Zdnull} \omega_n^2(\bsh,\bsz) \, \ralpha^{-1/\lambda}(\bsh)
      \qquad \forall \lambda \in [1/2,\alpha)
      .
      \numberthis \label{eq:initial_bound_existence}
    \end{align*}

    In order to take the average of $\big[\eran{d}(\randpstarz)\big]^2$ over all vectors in $G_{n,\tau}$, we first consider, for an arbitrary $\bsh$, the average of $\omega_n^2(\bsh,\bsz)$, as defined in \eqref{eq:omega-n-z}, over all vectors in $G_{n,\tau}$, hereby making use of Lemma~\ref{lemma:averaging_over_vecs},
    \begin{align*}
        &\frac{1}{|G_{n,\tau}|}\sum_{\bsz \in G_{n,\tau}} \omega_n^2(\bsh,\bsz)
        \\
        &=
        \frac{1}{|G_{n,\tau}|}\sum_{\bsz \in G_{n,\tau}} \frac{1}{|\Pn|^2} \Bigg(  \sum_{p \in \Pn}\bbone(\bsh \cdot \bsz \equiv_p 0) + \sum_{\substack{p,q \in \Pn \\p \neq q}}\bbone(\bsh \cdot \bsz \equiv_p 0) \, \bbone(\bsh \cdot \bsz \equiv_q 0) \Bigg)
        \\
        &=
        \frac{1}{|\Pn|^2}\Bigg(\sum_{p \in \Pn}\frac{1}{|G_\tau^{(p)}|}\sum_{\bsz \in G_\tau^{(p)}}\bbone(\bsh \cdot \bsz \equiv_p 0) \\
        &\hspace{20mm}+
        \sum_{\substack{p,q \in \Pn \\p \neq q}}\bigg(\frac{1}{|G_\tau^{(p)}|}\sum_{\bsz \in G_\tau^{(p)}}\bbone(\bsh\cdot\bsz \equiv_p 0)\bigg)\bigg(\frac{1}{|G_\tau^{(q)}|}\sum_{\bsz \in G_\tau^{(q)}}\bbone(\bsh\cdot\bsz \equiv_q 0)\bigg) \Bigg)
        \\
        &\leq
        \frac{1}{|\Pn|^2}\Bigg(\sum_{p \in \Pn} \bigg(\bbone(\bsh \equiv_p \bszero) + \frac{1}{\tau\,p} \bigg)
        \\
        &\hspace{20mm}+
        \sum_{\substack{p,q \in \Pn \\ p \neq q}}\bigg(\bbone(\bsh \equiv_{p \, q} \bszero) + \frac{1}{\tau\,p} \, \bbone(\bsh \equiv_q \bszero) + \frac{1}{\tau\,q} \, \bbone(\bsh \equiv_p \bszero) + \frac{1}{\tau^2\,p\,q} \bigg) \Bigg)
        .
    \end{align*}

    We now use this in the average of $\big[\eran{d}(\randpstarz)\big]^2$ over all vectors in $G_{n,\tau}$, continuing from~\eqref{eq:initial_bound_existence}, to obtain
    \begin{align*}
        &\frac{1}{|G_{n,\tau}|}\sum_{\bsz \in G_{n,\tau}} \left[\eran{d}(\randpstarz)\right]^2
        \\
        &\leq
        \frac{1}{B_{n,\tau}^{2-1/\lambda}} \sum_{\bsh \in \Zdnull} \bigg(\frac{1}{|G_{n,\tau}|}\sum_{\bsz \in G_{n,\tau}}\omega_n^2(\bsh,\bsz)\bigg) \, \ralpha^{-1/\lambda}(\bsh) \\
        &\leq
        \frac{1}{B_{n,\tau}^{2-1/\lambda} \, |\Pn|^2}\Bigg(\sum_{p \in \Pn} \bigg(\sum_{\bsh \in \Zdnull}\bigg(\bbone(\bsh \equiv_p \bszero)+ \frac{1}{\tau\,p} \bigg) \ralpha^{-1/\lambda}(\bsh) \bigg) \\
        &\hspace{15mm}+
        \sum_{\substack{p,q \in \Pn \\p \neq q}}\sum_{\bsh \in \Zdnull}\left(\bbone(\bsh \equiv_{p \, q} \bszero) + \frac{1}{\tau\,p} \, \bbone(\bsh \equiv_q \bszero) + \frac{1}{\tau\,q} \, \bbone(\bsh \equiv_p \bszero) + \frac{1}{\tau^2\,p\,q} \right) \ralpha^{-1/\lambda}(\bsh) \Bigg) \\
        &\leq
        \frac{\muquant}{B_{n,\tau}^{2-1/\lambda} \, |\Pn|^2} \Bigg(
        \sum_{p \in \Pn}
        \left( \frac{1}{p^{\alpha / \lambda}} + \frac{1}{\tau\,p} \right)
        +
        \sum_{\substack{p,q \in \Pn \\p \neq q}}
        \left(\frac{1}{(p \, q)^{\alpha / \lambda}} + \frac{1}{\tau\,p \, q^{\alpha / \lambda}}  + \frac{1}{\tau\,q \, p^{\alpha / \lambda}} + \frac{1}{\tau^2 \, p \, q} \right)
        \Bigg)
        \\
        &\leq
        \frac{\muquant}{B_{n,\tau}^{2-1/\lambda}} \left(\left(1+\frac{1}\tau\right) \frac{2 \ln n}{\constPn n^2} + \left(1+\frac{1}\tau\right)^2 \frac{4}{n^2}\right)
        ,
    \end{align*}
    where we have used $\alpha/\lambda > 1$, $p > n/2$, $q > n/2$ and~\eqref{eq:PNT_lowerbound},
    and for the third step we brought the sums over $\bsh$ further in and then used the property~\eqref{eq:ralpha-property} for the multiples and recombined the sums.

    If we fix $\bsz \in G_{n,\tau}$ to be the minimiser of $\eran{d}(\randpstarz)$, then, using~\eqref{eq:B_n} (and with the choice of $\lambda$ in~\eqref{eq:B_n} and~\eqref{eq:initial_bound_existence} fixed to be the same for convenience)
    we obtain
    \begin{align*}
        \exists \bsz^\star \in G_{n,\tau} : \qquad
        \eran{d}(\randpstarzstar)
        &\leq
        \frac{\Big[\muquant\Big]^{1/2}}{n \, B_{n,\tau}^{1-1/2\lambda}} \left[ \left(2+\frac{2}\tau\right) \left( \frac{\ln n}{\constPn} + 2 + \frac{2}\tau \right) \right]^{1/2} \\
        &\le
        \frac{\Big[4 \, \muquant\Big]^\lambda}{n^{\lambda+1/2}} \left[ \left( \frac{1}{1-\tau} \right)^{2\lambda} \frac{1-\tau^2}{2 \tau} \left(\frac{\ln n}{\constPn} + 2 + \frac{2}\tau\right) \right]^{1/2}
        .
    \end{align*}
    Finally we fix $\tau \in (0, 1)$ to be the minimiser of $(1/(1-\tau))^{2\lambda} \, (1-\tau^2)/(2\tau)$ which is $\tau^\star = (\sqrt{\lambda^2 + 2\lambda - 1} - \lambda) / (2\lambda - 1)$, for $\lambda > 1/2$, and bound this minimum by $3\lambda$, and we bound $2/\tau^\star \le 2(1+\sqrt{2})\lambda$. We then obtain
    \begin{align*}
        \exists \bsz^\star \in G_{n,\tau^\star} : \qquad
        \eran{d}(\randpstarzstar)
        &\le
        \frac{\Big[4 \, \muquant\Big]^\lambda}{n^{\lambda+1/2}}
        \Bigg[
          3 \lambda
          \left(\frac{\ln n}{\constPn} + 2 + 2(1+\sqrt{2}) \lambda \right)
        \Bigg]^{1/2}
        \\
        &\le
        \frac{(C_\lambda\,\ln n)^{1/2}}{n^{\lambda+1/2}} \,
        \Big[\muquant\Big]^\lambda
        ,
    \end{align*}
    for some $C_\lambda > 0$.
    We note that the final bound also holds for $\lambda = 1/2$ in which case the minimiser above is achieved for $\tau \to 1$.
\end{proof}

\section{Constructive method}\label{section:construction_of_random_vec}

In this section we propose a component-by-component (CBC) construction to
obtain a generating vector $\bsz$ which can be used in the randomised
algorithm $\randpstarz$, see~\eqref{eq:ran-p}, for randomly selected $p \in \Pn$.

We focus our discussion on the induction step for dimension $s = 2,\ldots,d$. We assume that $\bsz':= (z_1,\ldots,z_{s-1})$ is already
chosen and fixed. In the following, we will use the notation $\bsz =
(\bsz',z_s)$ and $\bsh = (\bsh',h_s)$ to separate out the $s$th component
of the vectors. We fix the first component $z_1$ to be $1$. This is due to
the fact that all choices of this component which are relatively prime to
all of the primes in $\Pn$ will give the same error which is also the
minimum, see~\eqref{eq:dim1} forthcoming.

We first show that the sets of good components
$\widetilde{G}_{s,\tau,\bsz'}^{(p)} =
\widetilde{G}_{s,\alpha,\bsgamma,\tau,\bsz'}^{(p)}$, for $s=1,\ldots,d$,
as defined in~\eqref{eq:Gtilde} are sufficiently large. This is analogous to Lemma~\ref{lemma:good_sets_are_large}.

\begin{lemma} \label{lemma:Gtilde_are_large}
For prime $p$, $\tau\in (0,1)$ and any $\bsz'\in\bbZ^{s-1}$, we have
\[
  |\widetilde{G}_{s,\tau,\bsz'}^{(p)}| \ge \lceil\tau p\rceil \ge 1.
\]
\end{lemma}
\begin{proof}
  See \hyperlink{proof:lemma:Gtilde_are_large}{appendix}.
\end{proof}

From the definition of the set of good components \eqref{eq:Gtilde} we also obtain the analogous result to Lemma~\ref{lemma:lower_bound_on_r}.

\begin{lemma}\label{lemma:lower_bound_on_r_cbc}
For prime $p$, any $\bsz' \in \bbZ^{s-1}$, $z_s \in
\widetilde{G}_{s,\tau,\bsz'}^{(p)}$ and $\bsh \in \Z^s$ with $h_s \ne 0$
and satisfying $\bsh \cdot \bsz \equiv_p 0$, we have
    $$
      \ralpha(\bsh)
      \geq
       \sup_{\lambda \in [1/2,\alpha)}
       p^{\lambda} \bigg[\frac{2}{1-\tau} \, \sum_{\substack{\bsh \in\bbZ^s \\ h_s\ne 0}}
 \ralpha^{-1/\lambda}(\bsh) \bigg]^{-\lambda}
       .
    $$
    Furthermore, if $p \in \Pn$ then
    $$
      \ralpha(\bsh)
      >
      \widetilde{B}_{s,n,\tau}
    $$
    where
    \begin{align}\label{eq:Btilde}
      \widetilde{B}_{s,n,\tau}
      =
      \widetilde{B}_{s,\alpha,\bsgamma,n,\tau}
      :=
      \sup_{\lambda \in [1/2,\alpha)} n^{\lambda} \bigg[\frac{4}{1-\tau} \, \sum_{\substack{\bsh \in\bbZ^s \\ h_s\ne 0}}
 \ralpha^{-1/\lambda}(\bsh) \bigg]^{-\lambda}
      .
    \end{align}
\end{lemma}

\begin{proof}
  See \hyperlink{proof:lemma:lower_bound_on_r_cbc}{appendix}.
\end{proof}

\subsection{CBC construction of the fixed generating vector}\label{section:construction_proof}

Similar to the CBC construction for the deterministic case, see~\eqref{eq:det-cbc}, we write for
the randomised error, see Proposition~\ref{prop:K_error},
\begin{align}\label{eq:cbc2}
 \left[\eran{s}(\randpstarzs)\right]^2
 &= \underbrace{
 \sum_{\bsh' \in \Z^{s-1} \setminus \{\bszero\}} \omega_n^2(\bsh', \bsz') \, \ralpha^{-2}(\bsh')
 }_{\left[\eran{s-1}(\randpstarsminusone)\right]^2}
 +
 \underbrace{\sum_{\bsh \in \Z^s,\, h_s \neq 0} \omega_n^2(\bsh,\bsz) \, \ralpha^{-2}(\bsh)
 }_{=:\,\thetaquant(z_s)}.
\end{align}
To determine $z_s$ we will work solely with the second sum
$\thetaquant(z_s)$ which depends on all previous components $\bsz' = (z_1,
\ldots, z_{s-1})$.

Since $z_s \in \Z_N$ with $N = \prod_{p \in \Pn} p$ contains too many
choices, we will exploit the isomorphism $z_s \cong \langle z_s^{(p_1)},
\ldots, z_s^{(p_L)}\rangle$ where $z_s^{(p)} \equiv_p z_s$
 for each
$p\in\Pn$, and choose $z_s^{(p)}$ prime by prime. The precise ordering of
the primes does not matter. For simplicity of exposition we will consider
the primes in increasing order.

From the definition of $\omega_n(\bsh,\bsz)$ in \eqref{eq:omega-n-z}, we can write
\begin{align} \label{eq:bad}
 \thetaquant(z_s)
 &= \frac{1}{|\Pn|^2}\sum_{p \in \Pn} \sum_{q\in\Pn}
     \sum_{\substack{\bsh \in \Z^s \\ h_s \neq 0 \\ \bsh\cdot\bsz \equiv_p 0 \\ \bsh\cdot\bsz \equiv_q 0}}
     \ralpha^{-2}(\bsh) \notag\\
 &= \frac{1}{|\Pn|^2} \sum_{p \in \Pn}
   \bigg[
    \underbrace{\sum_{\substack{\bsh \in \Z^s \\ h_s \neq 0 \\ \bsh\cdot\bsz \equiv_p 0}}
     \ralpha^{-2}(\bsh)}_{\theta_s^{(p)}(z_s^{(p)})}
     + 2 \sum_{\substack{q\in\Pn \\ q<p}}
     \sum_{\substack{\bsh \in \Z^s \\ h_s \neq 0 \\ h_s\not\equiv_p 0
     \\ \bsh\cdot\bsz \equiv_p 0 \\ \bsh\cdot\bsz \equiv_q 0}}
     \ralpha^{-2}(\bsh)
     + 2 \sum_{\substack{q\in\Pn \\ q<p}}
     \sum_{\substack{\bsh \in \Z^s \\ h_s \neq 0 \\ h_s\equiv_p 0
     \\ \bsh'\cdot\bsz' \equiv_p 0 \\ \bsh\cdot\bsz \equiv_q 0 }}
     \ralpha^{-2}(\bsh)
   \bigg],
\end{align}
where we used the symmetry in the double sum with respect to $p$ and $q$
to rewrite it as the sum over the ``diagonal'' $q=p$ plus twice the sum
over the ``triangle'' $q<p$, and then we further split the triangle part
into two, depending on whether or not $h_s\equiv_p 0$. Note that when
$h_s\equiv_p 0$, the condition $\bsh\cdot\bsz \equiv_p 0$ is effectively
independent of $z_s^{(p)}$ so we replaced it by $\bsh'\cdot\bsz'\equiv_p
0$ in~\eqref{eq:bad}.

We are now tempted to use the expression in-between the square brackets as
our search criterion to obtain $z_s^{(p)}$ for each $p\in\Pn$ in
increasing order. The first term is precisely $\theta_s^{(p)}(z_s^{p})$
which is the familiar search criterion for the CBC construction of a
deterministic lattice rule with~$p$ points. The second term depends on~$z_s^{(p)}$ as well as all ``past'' elements in the isomorphism~$z_s^{(q)}$ for $q<p$ which were already chosen, and therefore the
expression can be evaluated. The third term, however, depends only on past
elements $z_s^{(q)}$ for $q<p$, but not on the current element~$z_s^{(p)}$, and consequently, the search would have no control over the
size of this third term.

A remedy for this troublesome third term in \eqref{eq:bad} is as follows:
we swap the order of the sums for $p$ and $q$ and then interchange the
labeling of the two primes, to obtain
\begin{align*}
 \sum_{p \in \Pn}
 \sum_{\substack{q\in\Pn \\ q<p}}
 \sum_{\substack{\bsh \in \Z^s \\ h_s \neq 0 \\ h_s\equiv_p 0
 \\ \bsh'\cdot\bsz' \equiv_p 0\\ \bsh\cdot\bsz \equiv_q 0 }}
 \ralpha^{-2}(\bsh)
 = \sum_{q \in \Pn} \sum_{\substack{p\in\Pn \\ p>q}}
     \sum_{\substack{\bsh \in \Z^s \\ h_s \neq 0 \\ h_s\equiv_p 0
     \\ \bsh'\cdot\bsz' \equiv_p 0\\ \bsh\cdot\bsz \equiv_q 0 }}
     \ralpha^{-2}(\bsh)
 = \sum_{p \in \Pn} \sum_{\substack{q\in\Pn \\ q>p}}
     \sum_{\substack{\bsh \in \Z^s \\ h_s \neq 0 \\ h_s\equiv_q 0
     \\ \bsh\cdot\bsz \equiv_p 0 \\ \bsh'\cdot\bsz' \equiv_q 0}}
     \ralpha^{-2}(\bsh).
\end{align*}
Using this last expression in place of the third term in \eqref{eq:bad}
leads to
\begin{align*}
 \thetaquant(z_s)
 = \frac{1}{|\Pn|^2}\sum_{p \in \Pn} T_s^{(p)}(z_s^{(p)}),
\end{align*}
with
\begin{align} \label{eq:Td}
 T_s^{(p)}(z_s^{(p)})
 &= T_{s,\alpha,\bsgamma,n,\bsz'}^{(p)}(z_s^{(p)}; \{ z_s^{(q)} \}_{q \in \Pn ,\, q < p}) \notag\\
 &:= \underbrace{\sum_{\substack{\bsh \in \Z^s \\ h_s \neq 0 \\ \bsh\cdot\bsz \equiv_p 0}}
     \ralpha^{-2}(\bsh)}_{\theta_s^{(p)}(z_s^{(p)})}
     + 2 \sum_{\substack{q\in\Pn \\ q<p}}
     \sum_{\substack{\bsh \in \Z^s \\ h_s \neq 0 \\ h_s\not\equiv_p 0
     \\ \bsh\cdot\bsz \equiv_p 0 \\ \bsh\cdot\bsz \equiv_q 0}}
     \ralpha^{-2}(\bsh)
     + 2 \sum_{\substack{q\in\Pn \\ q>p}}
     \sum_{\substack{\bsh \in \Z^s \\ h_s \neq 0 \\ h_s\equiv_q 0 \\ \bsh\cdot\bsz \equiv_p 0 \\ \bsh'\cdot\bsz' \equiv_q 0}}
     \ralpha^{-2}(\bsh),
\end{align}
where now the third term does depend on $z_s^{(p)}$ as well as the
``future'' primes $q>p$, but conveniently it does not depend on the future
elements $z_s^{(q)}$ for $q>p$ which we are yet to choose. We can now use
this expression in a component-by-component and prime-by-prime manner to determine
the generating vector.

\begin{theorem}\label{thm:constructive_theorem}
Let $d \in \N$, $\alpha>1/2$, positive weights $\bsgamma = \{
\gamma_\setu \}_{\setu \subset \N}$, $\tau\in (0,1)$ and $n \ge 2$.
For each $s = 1,\ldots,d$ and then each prime $p\in\Pn$ in increasing
order, with the previously chosen components $\bsz'_\star =
(z_1^\star,\ldots,z_{s-1}^\star)$ and elements $z_{s,\star}^{(q)}$ for
$q<p$ held fixed, we choose $z_{s,\star}^{(p)}$ from the set
$\widetilde{G}^{(p)}_{s,\tau,\bsz'_\star}$ defined in \eqref{eq:Gtilde} to
minimise $T^{(p)}_s(z_s^{(p)}; \{ z_{s,\star}^{(q)} \}_{q \in \Pn ,\, q <
p})$ defined in~\eqref{eq:Td}. Then the randomised error for integration
of functions from the space $\spaceH$ using the randomised lattice
algorithm $\randpstarzstar$, with the such constructed generating vector
$\bsz^\star =(z_1^\star,\ldots,z_d^\star)$, has
$$
  \eran{d}(\randpstarzstar)
  \le
  \frac{\big(C_{\tau,\lambda}\,\ln n\big)^{1/2}}{n^{\lambda+1/2}}
  \Big[\muquant\Big]^\lambda
  \qquad
    \forall \lambda \in [1/2,\alpha)
  ,
$$
where $C_{\tau,\lambda} > 0$ goes to infinity for $\tau \to 0$ or for $\tau \to 1$.
    The quantity $\muquant$, as defined in~\eqref{eq:mu_quantity}, goes to infinity for $\lambda \to \alpha$.
    The upper bound is independent of the dimension if the weights $\bsgamma$ are such that
    $$
    \sum_{\emptyset \neq \fraku \subset \N} \gamma_{\fraku}^{1/\lambda} \left(2\zeta(\alpha/\lambda) \right)^{|\fraku|}
    <
    \infty
    .
    $$
\end{theorem}

\begin{proof}
For $s \ge 2$, suppose we have constructed the first $s-1$ components of the vector
$\bsz' = (z_1,\ldots,z_{s-1})$ and are now looking to determine $z_s \cong
\langle z_s^{(p_1)}, \ldots, z_s^{(p_L)}\rangle$. Our inductive hypothesis
is that
\begin{align*}
  \left[\eran{s-1}(\randpstarsminusone)\right]^2
  \le \frac{C_{\tau,\lambda}\,\ln n}{n^{2\lambda+1}}
  \bigg(\sum_{\bsh' \in \Z^{s-1}\setminus\{\bszero\}} \ralpha^{-1/\lambda}(\bsh')\bigg)^{2\lambda}
\end{align*}
for all $\lambda \in [1/2,\alpha)$. We will return to specify
$C_{\tau,\lambda}$ at the end.

As shown in \eqref{eq:cbc2}, we can write
\begin{align*}
 \left[\eran{s}(\randpstarzs)\right]^2
 &= \left[\eran{s-1}(\randpstarsminusone)\right]^2 + \thetaquant(z_s) \\
 &= \left[\eran{s-1}(\randpstarsminusone)\right]^2
 + \frac{1}{|\Pn|^2} \sum_{p \in \Pn} \Tquant{p}(z_s^{(p)}).
\end{align*}
For each $p$, we choose $z_{s,\star}^{(p)}$ from the ``good'' set
$\widetilde{G}^{(p)}_s$, defined in~\eqref{eq:Gtilde}, to minimise
$T_s^{(p)}(z_s^{(p)})$ defined in~\eqref{eq:Td}. Note that all previous
components and residues modulo~$q$, for $q < p$ in the $d$th dimension,
are the ones we previously selected, i.e., we choose $z_{s,\star}^{(p)}$
to minimise $T_s^{(p)}(z_s^{(p)}) = T_s^{(p)}(z_s^{(p)}; \{
z_{s,\star}^{(q)} \}_{q \in \Pn ,\, q < p})$ and $\bsz$ and $\bsz'$
in~\eqref{eq:Td} consist of the previously selected ones, except for~$z_s^{(p)}$. If we choose $z_{s,\star}^{(p)} \in \widetilde{G}^{(p)}_s$ to
minimise this, then it will be at least as good as the average, i.e.,
\begin{align*}
 T_s^{(p)}(z_{s,\star}^{(p)})
 \le \frac{1}{|\widetilde{G}^{(p)}_s|} \sum_{z_s^{(p)} \in \widetilde{G}^{(p)}_s} T_s^{(p)}(z_s^{(p)}).
\end{align*}
Thus, after doing this for every $p\in\Pn$, we can sum up the upper bounds
with respect to $p$ to obtain, with $z_{s,\star} \cong \langle
z_{s,\star}^{(p_1)}, \ldots, z_{s,\star}^{(p_L)}\rangle$,
\[
 \thetaquant(z_{s,\star})
 = \frac{1}{|\Pn|^2}\sum_{p \in \Pn} T_s^{(p)}(z_{s,\star}^{(p)})
 \le \frac{1}{|\Pn|^2}\sum_{p \in \Pn}
 \frac{1}{|\widetilde{G}^{(p)}_s|} \sum_{z_s^{(p)} \in \widetilde{G}^{(p)}_s} T_s^{(p)}(z_s^{(p)})
 =  A_1 + A_2 + A_3,
\]
with $A_1$, $A_2$, $A_3$ corresponding to averages with respect to the
three terms in \eqref{eq:Td}.

We have
\begin{align*}
 A_1
 &:= \frac{1}{|\Pn|^2}\sum_{p \in \Pn}
 \frac{1}{|\widetilde{G}^{(p)}_s|} \sum_{z_s^{(p)} \in \widetilde{G}^{(p)}_s} \theta_s^{(p)}(z_s^{(p)})
 \\
 &\,\le \frac{1}{|\Pn|^2}\sum_{p \in \Pn}
 \bigg(\frac{2}{(1-\tau)p}
 \sum_{\substack{\bsh \in \Z^s \\ h_s \ne 0}} \ralpha^{-1/\lambda}(\bsh)\bigg)^{2\lambda}
 \le \frac{\ln n}{n^{2\lambda+1}}
 \frac{2^{4\lambda}}{\constPn\,(1-\tau)^{2\lambda}}
 \bigg(\sum_{\substack{\bsh \in \Z^s \\ h_s \neq 0}} \ralpha^{-1/\lambda}(\bsh)\bigg)^{2\lambda},
\end{align*}
where in the first inequality we used the bound in \eqref{eq:Gtilde} which
holds for all $z_s^{(p)}\in \widetilde{G}_s^{(p)}$ and then the average
over $z_s^{(p)}$ dropped out; and in the second inequality we used
$1/|\Pn|\le \ln n/(\constPn \ln n)$ from~\eqref{eq:PNT_lowerbound} and $1/p < 2/n$.

Next we we have
\begin{align*}
 A_2
 &:= \frac{2}{|\Pn|^2}\sum_{p \in \Pn}
 \frac{1}{|\widetilde{G}^{(p)}_s|} \sum_{z_s^{(p)} \in \widetilde{G}^{(p)}_s}
 \sum_{\substack{q\in\Pn \\ q < p}}
 \sum_{\substack{\bsh \in \Z^s \\ h_s \neq 0 \\ h_s\not\equiv_p 0 \\
 \bsh\cdot\bsz \equiv_p 0\\ \bsh\cdot\bsz \equiv_q 0}}
 \ralpha^{-2}(\bsh)
 \le \frac{2}{|\Pn|^2}\sum_{p \in \Pn}
 \frac{1}{|\widetilde{G}^{(p)}_s|} \sum_{\substack{q\in\Pn \\ q < p}}
 \underbrace{\sum_{\substack{\bsh \in \Z^s \\ h_s \neq 0 \\ \bsh\cdot\bsz \equiv_q 0}}
 \ralpha^{-2}(\bsh)}_{\theta_s^{(q)}(z_s^{(q)})}
 \\
 &\,\le \frac{2}{|\Pn|^2}\sum_{p \in \Pn}
 \frac{1}{\tau p} \sum_{\substack{q\in\Pn \\ q < p}}
 \bigg(\frac{2}{(1-\tau)q} \sum_{\substack{\bsh \in \Z^s \\ h_s \neq 0}} \ralpha^{-1/\lambda}(\bsh)\bigg)^{2\lambda}
 \le
 \frac{1}{n^{2\lambda+1}} \frac{2^{4\lambda+1}}{\tau(1-\tau)^{2\lambda}}
 \bigg(\sum_{\substack{\bsh \in \Z^s \\ h_s \neq 0}} \ralpha^{-1/\lambda}(\bsh)\bigg)^{2\lambda},
\end{align*}
where in the first inequality we noted that for $h_s\not\equiv_p 0$ there
is at most one value of $z_s^{(p)}\in \widetilde{G}_s^{(p)}$ satisfying
$\bsh\cdot\bsz\equiv_p 0$ and then we dropped the condition
$h_s\not\equiv_p 0$; in the second inequality we applied the bound
\eqref{eq:Gtilde} since $z_s^{(q)}\in \widetilde{G}_s^{(q)}$ for each
$q<p$ and we used $|\widetilde{G}_s^{(p)}|\ge \tau p$; and in the final
inequality we used $1/p <2/n$, $1/q <2/n$ and then the double sum over the
primes dropped out.

It remains to tackle the most complicated term $A_3$. We have
\begin{align} \label{eq:A3}
 A_3
 &:= \frac{2}{|\Pn|^2}\sum_{p \in \Pn}
 \frac{1}{|\widetilde{G}^{(p)}_s|} \sum_{z_s^{(p)} \in \widetilde{G}^{(p)}_s}
 \sum_{\substack{q\in\Pn \\ q>p}}
 \sum_{\substack{\bsh \in \Z^s \\ h_s \neq 0 \\ h_s\equiv_q 0 \\ \bsh\cdot\bsz \equiv_p 0 \\ \bsh'\cdot\bsz' \equiv_q 0}}
 \ralpha^{-2}(\bsh)
 \notag\\
 &\,\le \frac{2}{|\Pn|^2}\sum_{p \in \Pn} \sum_{\substack{q\in\Pn \\ q > p}} \bigg[
  \sum_{\substack{\bsh \in \Z^s \\ h_s \neq 0 \\ h_s\equiv_p 0 \\ h_s\equiv_q 0 \\ \bsh'\cdot\bsz'\equiv_p 0 \\ \bsh'\cdot\bsz' \equiv_q 0}}
  \ralpha^{-2}(\bsh)
 + \frac{1}{|\widetilde{G}^{(p)}_s|}
  \sum_{\substack{\bsh \in \Z^s \\ h_s \neq 0 \\ h_s\not\equiv_p 0 \\ h_s\equiv_q 0 \\ \bsh'\cdot\bsz' \equiv_q 0}}
  \ralpha^{-2}(\bsh)
  \bigg],
\end{align}
where we used the properties that for $h_s\equiv_p 0$ the condition
$\bsh\cdot\bsz\equiv_p 0$ is independent of $z_s^{(p)}$ so the average
dropped out, while for $h_s\not\equiv_p 0$ there is at most one $z_s^{(p)}
\in \widetilde{G}^{(p)}_s$ satisfying $\bsh\cdot\bsz\equiv_p 0$.

To proceed further, one may be tempted to pull out factors of $p$ and $q$
from those $h_s$ which are multiples of $p$ and/or $q$. A quick
calculation reveals that this would not yield the correct convergence
order. To get the full convergence order we will need to make use of the
property $\ralpha(\bsh)> \widetilde{B}_{s,n,\tau}$, see \eqref{eq:Btilde},
which holds for all $\bsh$ satisfying $h_s\ne 0$ and
$\bsh\cdot\bsz\equiv_q 0$ provided that
$z_s^{(q)}\in\widetilde{G}^{(q)}_s$. Note that the expression inside the
square brackets in \eqref{eq:A3} does not explicitly depend on
$z_s^{(q)}$, but we can introduce $z_s^{(q)}$ artificially since
$h_s\equiv_q 0$. We obtain
\begin{align*}
 A_3
 &\le \frac{2}{|\Pn|^2}\sum_{p \in \Pn} \sum_{\substack{q\in\Pn \\ q > p}} \bigg[
  \sum_{\substack{\bsh \in \Z^s \\ h_s \neq 0 \\ h_s\equiv_p 0 \\ h_s\equiv_q 0 \\
  \bsh\cdot\bsz \equiv_q 0 \\ \ralpha(\bsh) > \widetilde{B}_{s,n,\tau}}}
  \ralpha^{-2}(\bsh)
 + \frac{1}{\tau p}
  \sum_{\substack{\bsh \in \Z^s \\ h_s \neq 0 \\ h_s\not\equiv_p 0 \\ h_s\equiv_q 0 \\
  \bsh\cdot\bsz \equiv_q 0 \\ \ralpha(\bsh) > \widetilde{B}_{s,n,\tau}}}
  \ralpha^{-2}(\bsh)
  \bigg],
\end{align*}
where we also dropped the condition $\bsh'\cdot\bsz'\equiv_p 0$ from the
first sum inside the square brackets and added the extra property
$\ralpha(\bsh)> \widetilde{B}_{s,n,\tau}$ to both sums. Note importantly
that we can only add this property to the sums before we make any
substitution of $\bsh$.

Now we use $1 \le (\ralpha(\bsh)/\widetilde{B}_{s,n,\tau})^{2-1/\lambda}$
inside the sums over $\bsh$ and then drop both conditions
$\bsh\cdot\bsz\equiv_q 0$ and $\ralpha(\bsh)> \widetilde{B}_{s,n,\tau}$,
to arrive at
\begin{align*}
 A_3
 &\le \frac{2}{|\Pn|^2} \sum_{p \in \Pn} \sum_{\substack{q \in \Pn \\ q>p}}
 \bigg[
 \sum_{\substack{\bsh \in \Z^s \\ h_s \neq 0 \\ h_s\equiv_p 0 \\ h_s\equiv_q 0}}
 \ralpha^{-2}(\bsh)
 \bigg(\frac{\ralpha(\bsh)}{\widetilde{B}_{s,n,\tau}}\bigg)^{2-1/\lambda}
 + \frac{1}{\tau p}
 \sum_{\substack{\bsh \in \Z^s \\ h_s \neq 0 \\ h_s\not\equiv_p 0 \\ h_s\equiv_q 0}}
 \ralpha^{-2}(\bsh)
 \bigg(\frac{\ralpha(\bsh)}{\widetilde{B}_{s,n,\tau}}\bigg)^{2-1/\lambda}
 \bigg]
 \\
 &\le \frac{1}{\widetilde{B}_{s,n,\tau}^{2-1/\lambda}} \frac{2}{|\Pn|^2} \sum_{p \in \Pn}
 \sum_{\substack{q \in \Pn \\ q>p}}
 \bigg[
 \sum_{\substack{\bsh \in \Z^s \\ h_s \neq 0 \\ h_s\equiv_p 0 \\ h_s\equiv_q 0}}
 \ralpha^{-1/\lambda}(\bsh)
 + \frac{1}{\tau p}
 \sum_{\substack{\bsh \in \Z^s \\ h_s \neq 0 \\ h_s\not\equiv_p 0\\ h_s\equiv_q 0}}
 \ralpha^{-1/\lambda}(\bsh)
 \bigg]
 .
\end{align*}
With the condition $\ralpha(\bsh)> \widetilde{B}_{s,n,\tau}$ properly
taken into account, we are now ready to pull out factors of $p$ and $q$.
Substituting $h_s = pq\, h_s^{\rm new}$ and $h_s = q\, h_s^{\rm new}$ in
the first and second sums over~$\bsh$, respectively, and then relabelling,
followed by collecting the resulting terms and applying the bound on
$\widetilde{B}_{s,n,\tau}$ in~\eqref{eq:Btilde}, we obtain
\begin{align*}
 A_3
 &\le \frac{1}{\widetilde{B}_{s,n,\tau_1}^{2-1/\lambda}} \frac{2}{|\Pn|^2} \sum_{p \in \Pn}
 \sum_{\substack{q \in \Pn \\ q>p}}
 \bigg[ \frac{1}{(pq)^{\alpha/\lambda}}
 \sum_{\substack{\bsh \in \Z^s \\ h_s \neq 0}}
 \ralpha^{-1/\lambda}(\bsh)
 + \frac{1}{\tau p\,q^{\alpha/\lambda}}
 \sum_{\substack{\bsh \in \Z^s \\ h_s \neq 0}}
 \ralpha^{-1/\lambda}(\bsh)
 \bigg]
 \\
 &\le
 \bigg(\frac{1}{n}\frac{4}{1-\tau}\sum_{\substack{\bsh \in \Z^s \\ h_s \neq 0}} \ralpha^{-1/\lambda}(\bsh)\bigg)^{2\lambda-1}
 \frac{4}{n^2} \bigg(1+\frac{1}\tau\bigg)
 \bigg(\sum_{\substack{\bsh \in \Z^s \\ h_s \neq 0}} \ralpha^{-1/\lambda}(\bsh)\bigg)
 \\
 &= \frac{2^{4\lambda} (1+\tau)}{\tau(1-\tau)^{2\lambda-1}} \frac{1}{n^{2\lambda+1}}
 \bigg(\sum_{\substack{\bsh \in \Z^s \\ h_s \neq 0}} \ralpha^{-1/\lambda}(\bsh)\bigg)^{2\lambda}.
\end{align*}

Combining all bounds, we conclude that for $z_{s,\star} \cong \langle
z_{s,\star}^{(p_1)}, \ldots, z_{s,\star}^{(p_L)}\rangle$ we have
\begin{align*}
 \thetaquant(z_{s,\star})
 &\le \frac{C_{\tau,\lambda}\,\ln n}{n^{2\lambda+1}}
  \bigg(\sum_{\substack{\bsh \in \Z^s \\ h_s \neq 0}} \ralpha^{-1/\lambda}(\bsh)\bigg)^{2\lambda},
\end{align*}
with
\begin{align*}
 C_{\tau,\lambda} :=
 \frac{2^{4\lambda}}{\constPn(1-\tau)^{2\lambda}}
 + \frac{2^{4\lambda+1}}{\tau(1-\tau)^{2\lambda}}
 + \frac{2^{4\lambda} (1+\tau)}{\tau(1-\tau)^{2\lambda-1}}.
\end{align*}

Using the induction hypothesis, we can now recombine using Jensen's
inequality to get
\begin{align*}
  \left[\eran{s}(\randpstarzstars)\right]^2
  &\leq
  \frac{C_{\tau,\lambda}\,\ln n}{n^{2\lambda+1}}
  \bigg( \sum_{\bsh' \in \Z^{s-1}\setminus\{\bszero\}} \ralpha^{-1/\lambda}(\bsh')\bigg)^{2\lambda}
  +
  \frac{C_{\tau,\lambda}\,\ln n}{n^{2\lambda+1}}
  \bigg(\sum_{\substack{\bsh \in \Z^s \\ h_s \neq 0}} \ralpha^{-1/\lambda}(\bsh)\bigg)^{2\lambda}
  \\
  &\leq
  \frac{C_{\tau,\lambda}\,\ln n}{n^{2\lambda+1}}
  \bigg( \sum_{\bsh \in \Z^{s} \setminus \{\bszero\}} \ralpha^{-1/\lambda}(\bsh)\bigg)^{2\lambda}
  .
\end{align*}
This completes the induction step of the proof.

Finally we return to verify the base case. By setting $z_1 = 1$, we can see that,
\begin{align*}
  \left[\eran{1}(\randpstarone)\right]^2
  &=
  \frac{1}{|\Pn|^2} \sum_{p,q \in \Pn} \sum_{\substack{h \in \Z\setminus\{0\} \\ h \equiv_p 0 \\ h \equiv_q 0}} \ralpha^{-2}(h)
  \\
  &=
  \frac{1}{|\Pn|^2} \sum_{p \in \Pn} \sum_{\substack{h \in \Z\setminus\{0\} \\ h \equiv_p 0}} \ralpha^{-2}(h) + \frac{1}{|\Pn|^2} \sum_{\substack{p,q \in \Pn \\ p \neq q}} \sum_{\substack{h \in \Z\setminus\{0\} \\ h \equiv_p 0 \\ h \equiv_q 0}} \ralpha^{-2}(h)
  \\
  &\leq
  \frac{1}{|\Pn|^2} \sum_{p \in \Pn} \frac{1}{p^{2\alpha}} \sum_{h \in \Z\setminus\{0\}} \ralpha^{-2}(h) + \frac{1}{|\Pn|^2} \sum_{\substack{p,q \in \Pn \\ p \neq q}} \frac{1}{(pq)^{2\alpha}}\sum_{h \in \Z\setminus\{0\}} \ralpha^{-2}(h)
  \\
  &\leq
  \frac{2^{2\alpha}\ln n}{\constPn n^{2\alpha+1}} \sum_{h \in \Z\setminus\{0\}} \ralpha^{-2}(h) + \frac{2^{4\alpha}}{n^{4\alpha}}\sum_{h \in \Z\setminus\{0\}} \ralpha^{-2}(h)
  \\
  &\leq
  \inf_{\lambda \in [1/2,\alpha)}
  \frac{C_\lambda' \ln n}{n^{2\lambda + 1}} \Bigg(\sum_{h \in \Z\setminus\{0\}} \ralpha^{-1/\lambda}(h)\Bigg)^{2 \lambda}
  \numberthis \label{eq:dim1}
\end{align*}
where
$$
C_\lambda'
=
\frac{2^{2\lambda}}{\constPn} + 2^{4 \lambda+1}
.
$$
This shows that the base case is satisfied and therefore completes the proof.
\end{proof}

Theorem~\ref{thm:constructive_theorem} leads to
Algorithm~\ref{alg:constructing_z} below which we summarise
as follows. Fix $z_1 = 1$. For each dimension $s = 2,\ldots,d$
and for each prime $p \in \Pn$, we compute
$\theta^{(p)}_s(z_s^{(p)})$ and $T_s^{(p)}(z_s^{(p)})$ for all values of
$z_s^{(p)} \in \Z_p$. We then use a partial sort to find the $\left\lceil
\tau p \right\rceil$ values of $z_s^{(p)}$ which have the lowest
$\theta^{(p)}_s(z_s^{(p)})$ and fix $z_{s,\star}^{(p)}$ to be one of these
values which has the minimal value of $T_s^{(p)}(z_s^{(p)})$. The
algorithm here deviates from the presentation in the proof as the values
of $z_s^{(p)} \in \Z_p$ are not taken from the set of good components
but from a subset of this set (being those elements which give the lowest
$\left\lceil \tau  p \right\rceil$ values of
$\theta^{(p)}_{s,\alpha,\bsgamma,\bsz'}(z_s)$). These are identical
to the sets used in~\cite{DGS2022}, and as mentioned at the end of
Section~\ref{section:randomised_algorithms}, they can be used
interchangeably in the proof.

\begin{alg}{Construction of a near optimal vector for budget $n \in \N$}\label{alg:constructing_z}
  \begin{algorithmic}
    \STATE Set $z_1 = 1$
    \FOR{$s=2,\ldots,d$}
        \FOR{$p \in \Pn$}
            \STATE Compute $\theta^{(p)}_s(z_s^{(p)}) = \theta^{(p)}_{s,\alpha,\bsgamma,\bsz'_\star}(z_s^{(p)})$ for all $z_s^{(p)} \in \Z_p$
            \STATE Compute $T_s^{(p)}(z_s^{(p)}) = T_{s,\alpha,\bsgamma,n,\bsz'_\star}^{(p)}(z_s^{(p)}; \{ z_{s,\star}^{(q)} \}_{q \in \Pn ,\, q < p})$ for all $z_s^{(p)} \in \Z_p$
            \STATE From the $\lceil \tau p \rceil$ best choices for $\theta^{(p)}_s(z_s)$, set $z_{s,\star}^{(p)}$ to be the one which minimises $T_s(z_s^{(p)})$
        \ENDFOR
    \ENDFOR
  \end{algorithmic}
\end{alg}

\subsection{Analysis of computation and complexity with product weights}\label{section:complexity_analysis}

We now look at the practical side of constructing generating vectors using Algorithm~\ref{alg:constructing_z} and analyse the asymptotic complexity of this algorithm. We will restrict our attention to product weights. These are defined as a sequence $\{\gamma_j\}_{j=1}^\infty$ and for $\fraku \subset \N$ we set
$$
\gamma_\fraku
=
\prod_{j \in \fraku} \gamma_j
,
$$
hence, for product weights, we have
$$
\ralpha(\bsh) = \prod_{j \in \supp(\bsh)} \gamma_j^{-1} \, |h_j|^\alpha
.
$$

\begin{proposition}
  For $d \in \N$, $\alpha > 1/2$, positive product weights $\{\gamma_j\}_{j=1}^\infty$, $\tau \in (0,1)$ and $n \ge 2$, Algorithm~\ref{alg:constructing_z} constructs the vector $\bsz^\star \in \Z^d$ that satisfies the bound of Theorem~\ref{thm:constructive_theorem} in time $\calO(d n^4 / \ln n)$ and using $\calO(n^4 / (\ln n)^2)$ memory. The generating vector can be used for any $d' \le d$, satisfying the bound of Theorem~\ref{thm:constructive_theorem} for that number of dimensions.
  The memory cost can be reduced to $\calO(n)$ at the expense of increasing the calculation cost by a factor~$d$.
\end{proposition}
\begin{proof}
From Algorithm~\ref{alg:constructing_z} we see that we have to compute the quantities $\theta^{(p)}_s(z_s^{(p)})$ and $T_s^{(p)}(z_s^{(p)})$ for all $z_s^{(p)} \in \Z_p$.

Let us first define the function for $x \in [0,1]$ and $\alpha > 1/2$
$$
\sigma_\alpha(x)
:=
\sum_{h \in \Z \setminus \{0\}}\frac{\Euler{}{h x}}{|h|^{2 \alpha}}
.
$$
We can see that, for integer $\alpha$,
$$
\sigma_\alpha(x)
=
\frac{(-1)^{\alpha+1} (2 \pi)^{2 \alpha}}{(2 \alpha)!} B_{2 \alpha}(x \bmod 1)
,
$$
where $B_{2 \alpha}$ is the Bernoulli polynomial of degree $2 \alpha$.
We will need these values for all $x = k/p$ with $k \in \Z_p$ for each $p \in \Pn$, and for all $x = (k/p + \ell/q) \bmod{1}$ with $k \in \Z_p$ and $\ell \in \Z_q$ for all $p, q \in \Pn$ with $p \ne q$. The last set is equivalent to $x = k' / (pq)$ for all $k' \in \Z_{pq}$ since $p$ and $q$ are relatively prime.
These values can be precomputed at the start of the construction and will use $\calO(n^4 / (\ln n)^2)$ memory.

For $\theta^{(p)}_s(z_s)$ we use the same recursive structure as for the standard fast CBC algorithm:
\begin{align}
  \theta^{(p)}_s(z_s)
  &=
  \sum_{\substack{\bsh\in\bbZ^s \\ h_s \ne 0 \\ \bsh\cdot\bsz\equiv_p 0}}
  \ralpha^{-2}(\bsh)
  =
  \sum_{\substack{\bsh\in\bbZ^s \\ h_s \ne 0}} \frac{1}{p} \sum_{k = 0}^{p-1} \rme^{2 \pi \rmi k \bsh \cdot \bsz / p} \,
  \ralpha^{-2}(\bsh)
  \nonumber \\
  &=
  \sum_{\substack{\bsh\in\bbZ^s \\ h_s \ne 0}} \frac{1}{p} \sum_{k = 0}^{p-1} \frac{\gamma_s^2 \, \rme^{2 \pi \rmi k h_s z_s / p}}{|h_s|^{2 \alpha}} \prod_{\substack{j = 1 \\ h_j \neq 0}}^{s-1} \frac{\gamma_j^2 \, \rme^{2 \pi \rmi k h_j z_j / p}}{|h_j|^{2 \alpha}}
  \nonumber \\
  &=
  \frac{1}{p} \sum_{k = 0}^{p-1} \gamma_s^2 \, \Bigg(\sum_{h \in \Z \setminus \{0\}}\frac{\rme^{2 \pi \rmi k h z_s / p}}{|h|^{2 \alpha}}\Bigg) \prod_{j = 1}^{s-1} \Bigg( 1 + \gamma_j^2  \sum_{h \in \Z \setminus \{0\}} \frac{\rme^{2 \pi \rmi k h z_j / p}}{|h|^{2 \alpha}} \Bigg)
  \nonumber \\
  &=
  \frac{\gamma_s^2}{p} \sum_{k = 0}^{p-1} \sigma_\alpha(k z_s / p)
  \underbrace{\prod_{j = 1}^{s-1} \left( 1 + \gamma_j^2 \, \sigma_\alpha(k z_j / p)\right)}_{=: P^{(p)}_{s-1}(k)}
  .
  \label{eq:compute_theta}
\end{align}
The last form can be recognised as a convolution and hence we can use an FFT method to calculate $\theta^{(p)}_s(z_s)$ for all $z_s \in \Z_p$ in $\calO(n \ln n)$, see~\cite{NC2006-1,NC2006-2}.
The storage of the products $P^{(p)}_{s-1}(k)$ for all $k \in \Z_p$ requires $\calO(n)$ memory, which over all $p \in \Pn$ requires thus $\calO(n^2/\ln n)$ memory.

To calculate $T_s^{(p)}(z_s^{(p)})$ we first recall its definition~\eqref{eq:Td},
$$
T_s^{(p)}(z_s^{(p)})
=
\theta_s^{(p)}(z_s^{(p)})
    + 2 \sum_{\substack{q\in\Pn \\ q<p}}
    \sum_{\substack{\bsh \in \Z^s \\ h_s \neq 0 \\ h_s\not\equiv_p 0
    \\ \bsh\cdot\bsz \equiv_p 0 \\ \bsh\cdot\bsz \equiv_q 0}}
    \ralpha^{-2}(\bsh)
    + 2 \sum_{\substack{q\in\Pn \\ q>p}}
    \sum_{\substack{\bsh \in \Z^s \\ h_s \neq 0 \\ h_s\equiv_q 0 \\ \bsh\cdot\bsz \equiv_p 0 \\ \bsh'\cdot\bsz' \equiv_q 0}}
    \ralpha^{-2}(\bsh)
.
$$
To simplify the computation of $T_j^{(p)}(z_j^{(p)})$, we will define $\widehat{T}_s^{(p)}(z_s^{(p)})$, which simply drops the condition $h_s \not\equiv_p 0$ on the second term, but does not alter the order of the $z_s^{(p)}$ values:
\begin{align*}
  \widehat{T}_s^{(p)}(z_s^{(p)})
  &:=
  \theta_s^{(p)}(z_s^{(p)})
     + 2 \sum_{\substack{q\in\Pn \\ q<p}}
     \sum_{\substack{\bsh \in \Z^s \\ h_s \neq 0
     \\ \bsh\cdot\bsz \equiv_p 0 \\ \bsh\cdot\bsz \equiv_q 0}}
     \ralpha^{-2}(\bsh)
     + 2 \sum_{\substack{q\in\Pn \\ q>p}}
     \sum_{\substack{\bsh \in \Z^s \\ h_s \neq 0 \\ h_s\equiv_q 0 \\ \bsh\cdot\bsz \equiv_p 0 \\ \bsh'\cdot\bsz' \equiv_q 0}}
     \ralpha^{-2}(\bsh)
  .
\end{align*}
We have already computed the first term in~\eqref{eq:compute_theta}.
The second term can be written as
\begin{align*}
  &2 \sum_{\substack{q\in\Pn \\ q<p}} \sum_{\substack{\bsh \in \Z^s \\ h_s \neq 0 \\ \bsh\cdot\bsz \equiv_p 0 \\ \bsh\cdot\bsz \equiv_q 0}} \ralpha^{-2}(\bsh)
  =
  2 \sum_{\substack{q\in\Pn \\ q<p}} \sum_{\substack{\bsh \in \Z^s \\ h_s \neq 0}} \frac{1}{p} \sum_{k = 0}^{p-1} \frac{1}{q} \sum_{\ell = 0}^{q-1} \Euler{}{k \bsh \cdot \bsz / p} \, \Euler{}{\ell \bsh \cdot \bsz / q} \, \ralpha^{-2}(\bsh)
  \\
  &\hspace{10mm} =
  2 \sum_{\substack{q\in\Pn \\ q<p}}
  \frac{1}{q} \sum_{\ell = 0}^{q-1}
  \frac{\gamma_s^2}{p} \sum_{k = 0}^{p-1}
  \sigma_\alpha\!\left(\frac{k z_s^{(p)}}{p} + \frac{\ell z_s^{(q)}}{q}\right)
  \underbrace{\prod_{j = 1}^{s-1}\left( 1 + \gamma_j^2 \, \sigma_\alpha\!\left(\frac{k z_j^{(p)}}{p} + \frac{\ell z_j^{(q)}}{q} \right) \right)}_{=: P^{(p,q)}_{s-1}(k, \ell)}
  .
\end{align*}
The innermost summation over $k$ can again be recognised as a convolution and hence can be performed via an FFT method for all $z_s^{(p)} \in \Z_p$ in $\calO(n \ln n)$.
So the cost of computing this full expression for all values of $z_s^{(p)} \in \Z_p$ is $\calO(n^3)$.
The cost of storing all of the products $P^{(p,q)}_{s-1}(k, \ell)$ requires $\calO(n^4/(\ln n)^2)$ memory.

Finally, the third term is
\begin{align*}
  &2 \sum_{\substack{q\in\Pn \\ q>p}} \sum_{\substack{\bsh \in \Z^s \\ h_s \neq 0 \\ h_s\equiv_q 0 \\ \bsh\cdot\bsz \equiv_p 0 \\ \bsh'\cdot\bsz' \equiv_q 0}} \ralpha^{-2}(\bsh)
  \\
  &\hspace{10mm} =
  2 \sum_{\substack{q\in\Pn \\ q>p}}
  \frac{1}{q^{2 \alpha}} \sum_{\substack{\bsh \in \Z^s \\ h_s \neq 0}}
  \bbone(q h_s z_s^{(p)} \equiv_p -\bsh' \cdot \bsz') \, \bbone(\bsh' \cdot \bsz' \equiv_q 0) \, \ralpha^{-2}(\bsh)
  \\
  &\hspace{10mm} =
  2 \sum_{\substack{q\in\Pn \\ q>p}}
  \frac{1}{q^{2 \alpha}} \sum_{\substack{\bsh \in \Z^s \\ h_s \neq 0}}
  \frac{1}{p} \sum_{k = 0}^{p-1}
  \frac{1}{q} \sum_{\ell = 0}^{q-1}
  \frac{\gamma_s^2 \, \Euler{}{k q h_s z_s^{(p)}/p}}{|h_s|^{2 \alpha}}
  \prod_{\substack{j = 1 \\ h_j \neq 0}}^{s-1} \Euler{}{h_j (k z_j^{(p)} / p + \ell z_j^{(q)} / q)} \frac{\gamma_j^2}{|h_j|^{2 \alpha}}
  \\
  &\hspace{10mm} =
  2 \sum_{\substack{q\in\Pn \\ q>p}}
  \frac{\gamma_s^2}{q^{2 \alpha+1} \, p} \sum_{k = 0}^{p-1}
  \sigma_\alpha\!\left(\frac{k \, q \, z_s^{(p)}}{p}\right)
  \left[ \sum_{\ell = 0}^{q-1} \prod_{j=1}^{s-1} \left(1 + \gamma_j^2 \, \sigma_\alpha\!\left(\frac{k z_j^{(p)}}{p} + \frac{\ell z_j^{(q)}}{q}\right)\right) \right]
  .
\end{align*}
Note that the product in the last line is the quantity $P^{(p,q)}_{s-1}(k,\ell)$.
Again, the sum over $k$ can be recognised as a convolution and hence can be performed via an FFT method for all $z_s^{(p)} \in \Z_p$ in $\calO(n\ln n)$.
Computing the square brackets as a vector for each value of $k \in \Z_p$ however takes~$\calO(n^2)$ and thus dominates the $\calO(n \ln n)$.
The total cost for computing this part for $z_s^{(p)} \in \Z_p$ hence is $\calO(n^3 / \ln n)$.
\end{proof}

\section{Numerical results} \label{section:numerical_results}

We now present numerical results for the new Algorithm~\ref{alg:constructing_z}.
For $\alpha = 1$ and $\alpha = 2$, $d=5$ and product weights $\gamma_j = j^{-3}$, we show in Figure~\ref{fig:DetRanError} a log-log plot of $\eran{d}(\randpstarzstarPn)$ against the number of function evaluations $n$ (red) with the generating vector $\bsz_{\Pn}^\star$ chosen via Algorithm~\ref{alg:constructing_z} and $\eran{d}(Q_{d,n,\bsz_n^\star}) = \edet{d}(Q_{d,n,\bsz_n^\star})$ against the number of function evaluations $n$ (blue) with the generating vector chosen via the deterministic CBC construction \cite{K2003,NC2006-1}.

\begin{figure}[H]
  \centering
  \begin{subfigure}{.5\linewidth}
    \centering
    \begin{tikzpicture}[scale=0.8, every node/.style={scale=0.6}]
      \begin{loglogaxis}[
          label style={font=\Large},
          xlabel = $n$,
          ylabel = $\eran{d}(\cdot)$
      ]
          \addplot[
              mark = *,
              mark size = 1.5pt,
              color = red,
              mark color = red
          ]
          table[y=ERan]
          {DetRanErrorVals1.txt};
          \label{plot:ERan1}

          \addplot+[
              color = red,
              mark = none,
              dashed
          ]
          table[y=RanOptLine]
          {DetRanErrorVals1.txt};
          \label{plot:RandOptLine1}

          \addplot[
              mark = *,
              mark size = 1.5pt,
              color = blue,
              mark color = blue
          ]
          table[y=EDet]
          {DetRanErrorVals1.txt};
          \label{plot:EDet1}

          \addplot+[
              color = blue,
              mark = none,
              dashed
          ]
          table[y=DetOptLine]
          {DetRanErrorVals1.txt};
          \label{plot:DetOptLine1}

          \addplot[red, dashed] coordinates {
              (337,0.0016889270952836744)
              (337,0.0009603587910222988)
          };
          \addplot[red, dashed] coordinates {
              (337,0.0009603587910222988)
              (491,0.0009603587910222988)
          };
          \node [red] at (axis cs:301,0.0013) {$1.5$};
          \node [red] at (axis cs:400,0.00079) {$1$};

          \addplot[blue, dashed] coordinates {
              (337,0.007144936226081703)
              (337,0.0049039582651518)
          };
          \addplot[blue, dashed] coordinates {
              (337,0.0049039582651518)
              (491,0.0049039582651518)
          };
          \node [blue] at (axis cs:310,0.006) {$1$};
          \node [blue] at (axis cs:400,0.0041) {$1$};
      \end{loglogaxis}
      \readrecordarray{DetRanErrorGradients1.txt}\data
      \node [draw = gray!40,fill=white,thin] at (2.5,0.7) {
              \shortstack[l]{
              \ref{plot:EDet1} $\eran{d}(Q_{d,n,\bsz^\star_n})$ : gradient = $\data[2]$ \\
              \ref{plot:ERan1} $\eran{d}(\randpstarzstarPn)$ : gradient = $\data[1]$
              }
          };
    \end{tikzpicture}
    \caption{$\alpha = 1$, $d = 5$, $\gamma_j = j^{-3}$}
  \end{subfigure}%
  \begin{subfigure}{.5\linewidth}
    \centering
    \begin{tikzpicture}[scale=0.8, every node/.style={scale=0.6}]
      \begin{loglogaxis}[
          label style={font=\Large},
          xlabel = $n$,
          ylabel = $\eran{d}(\cdot)$
      ]
          \addplot[
              mark = *,
              mark size = 1.5pt,
              color = red,
              mark color = red
          ]
          table[y=ERan]
          {DetRanErrorVals2.txt};
          \label{plot:ERan2}

          \addplot+[
              color = red,
              mark = none,
              dashed
          ]
          table[y=RanOptLine]
          {DetRanErrorVals2.txt};
          \label{plot:RandOptLine2}

          \addplot[
              mark = *,
              mark size = 1.5pt,
              color = blue,
              mark color = blue
          ]
          table[y=EDet]
          {DetRanErrorVals2.txt};
          \label{plot:EDet2}

          \addplot+[
              color = blue,
              mark = none,
              dashed
          ]
          table[y=DetOptLine]
          {DetRanErrorVals2.txt};
          \label{plot:DetOptLine2}

          \addplot[red, dashed] coordinates {
              (1759,2.4824138176607193e-07)
              (1759,9.917039120257037e-08)
          };
          \addplot[red, dashed] coordinates {
              (1759,9.917039120257037e-08)
              (2539,9.917039120257037e-08)
          };
          \node [red] at (axis cs:1560,1.56901864e-7) {$2.5$};
          \node [red] at (axis cs:2113,7.5e-08) {$1$};

          \addplot[blue, dashed] coordinates {
              (1759,1.7193153037466141e-06)
              (1759,8.252049150746002e-07)
          };
          \addplot[blue, dashed] coordinates {
              (1759,8.252049150746002e-07)
              (2539,8.252049150746002e-07)
          };
          \node [blue] at (axis cs:1620,1.1911286409178e-6) {$2$};
          \node [blue] at (axis cs:2113,6.2e-07) {$1$};
      \end{loglogaxis}
      \readrecordarray{DetRanErrorGradients2.txt}\data
      \node [draw = gray!40,fill=white,thin] at (2.5,0.7) {
              \shortstack[l]{
              \ref{plot:EDet2} $\eran{d}(Q_{d,n,\bsz^\star_n})$ : gradient = $\data[2]$ \\
              \ref{plot:ERan2} $\eran{d}(\randpstarzstarPn)$ : gradient = $\data[1]$
              }
          };
    \end{tikzpicture}
    \caption{$\alpha = 2$, $d = 5$, $\gamma_j = j^{-3}$}
  \end{subfigure}
  \caption{Comparison of the new randomised algorithm (red) with a deterministic lattice rule (blue) for the randomised error.}\label{fig:DetRanError}
\end{figure}
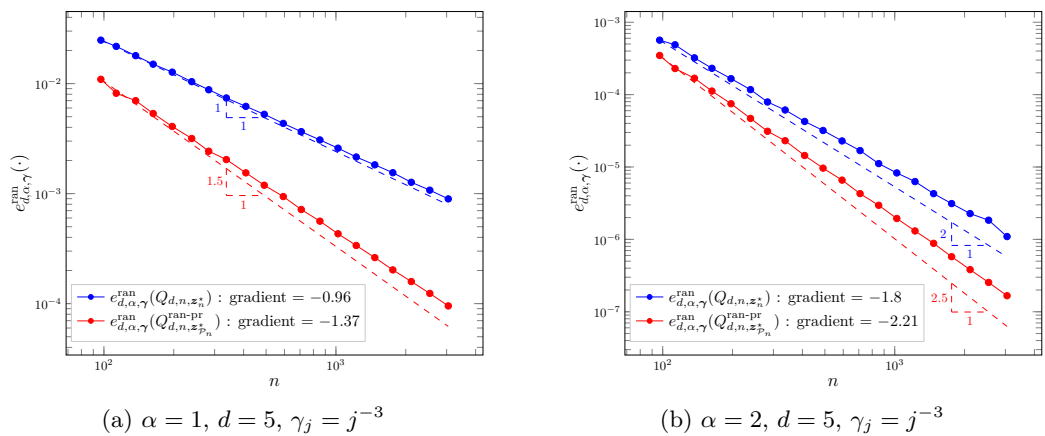

The range of values of $n$ are given by the closest primes to $(1.2)^k$ for $k = 25, \ldots, 44$ and the least squares fit of the gradient of the plotted points is shown.
We expect to see a convergence of~$n^{-\alpha-1/2}$ for the randomised algorithm, and a convergence of~$n^{-\alpha}$ for the deterministic algorithm. These are denoted by the dashed lines on the graphs.
Both graphs demonstrate that the order of convergence of the randomised algorithm exceeds that of the deterministic algorithm, as expected, but the asymptotic rate is not yet achieved for this range of~$n$ (which is also as expected).

\bigskip

\textbf{Acknowledgements.}\quad \mbox{}
We would like to acknowledge the support from the Research Foundation Flanders (FWO G091920N) and the Australian Research Council (DP21010083).

\appendix
\section{Appendix}

\begin{proof}[\hypertarget{proof:prop:edet}{Proof of Proposition~\ref{prop:edet}}]
    For an arbitrary $f \in \spaceH$ we see from~\eqref{eq:lattice_rule_error} that
    \begin{align}
        \nonumber
        |Q_{d,n,\bsz}(f)-I_d(f)|
        &\leq
        \sum_{\substack{\bsh\in \Zdnull \\ \bsh\cdot\bsz \equiv_n 0}} |\hat{f}(\bsh)|
        \\\nonumber
        &\leq
        \Bigg(\sum_{\substack{\bsh\in \Zdnull \\ \bsh\cdot\bsz \equiv_n 0}} \ralpha^{-2}(\bsh) \Bigg)^{1/2} \Bigg(\sum_{\substack{\bsh\in \Zdnull \\ \bsh\cdot\bsz \equiv_n 0}}  \ralpha^2(\bsh) \, |\hat{f}(\bsh)|^2 \Bigg)^{1/2}
        \\\label{eq:Qub}
        &\leq
        \Bigg(\sum_{\substack{\bsh\in \Zdnull \\ \bsh\cdot\bsz \equiv_n 0}} \ralpha^{-2}(\bsh) \Bigg)^{1/2} \norm{f}
        ,
    \end{align}
    where the second inequality is due to Cauchy--Schwarz.
    We see that the function
    $$
    f(\bsx)
    =
    \sum_{\bsh \in \Zdnull}
    \ralpha^{-2}(\bsh) \, \rme^{2 \pi \rmi \bsh \cdot \bsx}
    \qquad
    \in \qquad \spaceH
    $$
    attains equality for the upper bound on the error~\eqref{eq:Qub} and hence we have an explicit expression for the deterministic worst case error of a lattice rule.
\end{proof}

\begin{proof}[\hypertarget{proof:lemma:averaging_over_vecs}{Proof of Lemma~\ref{lemma:averaging_over_vecs}}]
    For the first part we note that when all of the components of $\bsh$ are multiples of $p$ then clearly any possible $\bsz \in \Z_p^d$ is a solution to the congruence. On the other hand, if there is at least one component which is not a multiple of $p$, say $h_i$, then we can see that there is exactly one possibility for $z_i \in \Z_p$ which solves the congruence
    $$
    h_i z_i
    \equiv_p
    -\sum_{\substack{j = 1 \\ j \neq i}}^d h_j z_j
    ,
    $$
    for all $p^{d-1}$ possible values of the $z_j \in \Z_p$ for $j \neq i$.
    Therefore, there will be exactly $p^{d-1}$ solutions to the congruence over all the values of $\bsz \in \Z_p^d$.

    We, therefore, have
    \begin{align*}
        \frac{1}{p^d}\sum_{\bsz \in \Z_p^d} \bbone(\bsh\cdot\bsz \equiv_p 0)
        &=
        \bbone(\bsh \equiv_p \bszero) + \frac{1}{p} \big(1-\bbone(\bsh \equiv_p \bszero)\big)
        \\
        &=
        \frac{p-1}{p}\bbone(\bsh \equiv_p \bszero) + \frac{1}{p}
        .
    \end{align*}

    The second part comes from the observation that from the above $p^{d-1}$ solutions it could be that $\calZ$ does not include all of them and hence
    \begin{align*}
        \frac{1}{|\calZ|} \sum_{\bsz \in \calZ} \bbone(\bsh\cdot\bsz \equiv_p 0)
        &\leq
        \bbone(\bsh \equiv_p \bszero) + \frac{p^{d-1}}{\ceil{\tau p^d}} \, \big(1-\bbone(\bsh \equiv_p \bszero)\big) \\
        &\leq
        \frac{\tau p - 1}{\tau p}\bbone(\bsh \equiv_p \bszero) + \frac{1}{\tau p}
        .
        \qedhere
    \end{align*}
\end{proof}

\begin{proof}[\hypertarget{proof:proposition:deterministic_bound_is_achieved}{Proof of Proposition~\ref{proposition:deterministic_bound_is_achieved}}]
    We fix $1/2 \leq \lambda < \alpha$ and take the average
        \begin{align*}
            \frac{1}{p^d}\sum_{\bsz \in \Z_p^d} \left[\edet{d}(Q_{d,p,\bsz})\right]^{1/\lambda}
            &= \frac{1}{p^d}\sum_{\bsz \in \Z_p^d} \Bigg(\sum_{\substack{\bsh \in \Zdnull \\ \bsh \cdot \bsz \equiv_{p} 0}} \ralpha^{-2}(\bsh)\Bigg)^{1/2\lambda} \\
            &\leq \frac{1}{p^d}\sum_{\bsz \in \Z_p^d} \sum_{\substack{\bsh \in \Zdnull \\ \bsh \cdot \bsz \equiv_{p} 0}} \ralpha^{-1/\lambda}(\bsh) \\
            &= \sum_{\bsh \in \Zdnull} \bigg(\frac{1}{p^d}\sum_{\bsz \in \Z_p^d} \bbone(\bsh \cdot \bsz \equiv_{p} 0)\bigg) \ralpha^{-1/\lambda}(\bsh) \\
            &= \sum_{\bsh \in \Zdnull} \bbone(\bsh \equiv_{p} \bszero) \, \ralpha^{-1/\lambda}(\bsh) +\frac{1}{p}\sum_{\bsh \in \Zdnull} \bbone(\bsh \not\equiv_{p} \bszero) \, \ralpha^{-1/\lambda}(\bsh) \\
            &\leq \sum_{\bsh \in \Zdnull}   \left(p^{|\supp(\bsh)|}\right)^{-\alpha/\lambda}\ralpha^{-1/\lambda}(\bsh) +\frac{1}{p} \, \muquant \\
            &\leq \frac{2}{p} \, \muquant
            .
    \end{align*}
    On the fourth line of the above, we used Lemma~\ref{lemma:averaging_over_vecs} and on the fifth line $\bsh$ was replaced with $p \bsh$ after which we used~\eqref{eq:ralpha-property}.

    For the second part of the proof, choose $\bsz^\star$ to be the vector which minimises the value of~$\edet{d}(Q_{d,p,\bsz})$. Therefore, for any $\lambda > 0$ the minimal value to the power $1/\lambda$ will be at least as good as the average of those values, and thus certainly
    $$
    \left[\edet{d}(Q_{d,p,\bsz^\star})\right]^{1/\lambda}
    \leq
    \frac{1}{p^d}\sum_{\bsz \in \Z_p^d} [\edet{d}(Q_{d,p,\bsz})]^{1/\lambda}
    \qquad \forall \lambda \in [1/2,\alpha)
    ,
    $$
    and the result holds for such a $\bsz^\star$.
\end{proof}

\begin{proof}[\hypertarget{proof:proposition:bound_on_theta}{Proof of Proposition~\ref{proposition:bound_on_theta}}]
    We take the average
    \begin{align*}
        \frac{1}{p} \sum_{z_s \in \Z_p} \left[\theta_s^{(p)}(z_s)\right]^{1/2\lambda}
        &\leq
        \frac{1}{p} \sum_{z_s \in \Z_p} \sum_{\substack{\bsh\in\bbZ^{s} \\ h_s \ne 0 \\ \bsh\cdot\bsz\equiv_p 0}}
        \ralpha^{-1/\lambda}(\bsh)
        \\
        &=
        \frac{1}{p} \sum_{z_s \in \Z_p} \sum_{\substack{\bsh\in\bbZ^{s} \\ h_s \ne 0 \\ h_s \not\equiv_p 0 \\ \bsh\cdot\bsz\equiv_p 0}}
        \ralpha^{-1/\lambda}(\bsh)
        +
        \frac{1}{p} \sum_{z_s \in \Z_p} \sum_{\substack{\bsh\in\bbZ^{s} \\ h_s \ne 0 \\ h_s \equiv_p 0 \\ \bsh\cdot\bsz\equiv_p 0}}
        \ralpha^{-1/\lambda}(\bsh)
        \\
        &=
        \frac{1}{p} \sum_{\substack{\bsh\in\bbZ^{s} \\ h_s \ne 0 \\ h_s \not\equiv_p 0}}
        \ralpha^{-1/\lambda}(\bsh)
        +
        \sum_{\substack{\bsh\in\bbZ^{s} \\ h_s \ne 0 \\ h_s \equiv_p 0 \\ \bsh'\cdot\bsz'\equiv_p 0}}
        \ralpha^{-1/\lambda}(\bsh)
        \\
        &\le
        \frac{1}{p} \sum_{\substack{\bsh\in\bbZ^{s} \\ h_s \ne 0 \\ h_s \not\equiv_p 0}}
        \ralpha^{-1/\lambda}(\bsh)
        +
        \frac{1}{p^{\alpha/\lambda}} \sum_{\substack{\bsh\in\bbZ^{s} \\ h_s \ne 0 \\ \bsh'\cdot\bsz'\equiv_p 0}}
        \ralpha^{-1/\lambda}(\bsh)
        \\
        &\le
        \frac{2}{p} \sum_{\substack{\bsh\in\bbZ^{d} \\ h_s \ne 0}}
        \ralpha^{-1/\lambda}(\bsh)
        .
    \end{align*}
    For the second part, in a similar manner to the proof of Proposition~\ref{proposition:deterministic_bound_is_achieved} we can see that the optimal value of $\theta_s^{(p)}(z_s)$ must be less than the average and must satisfy the bound for all $\lambda \in (1/2,\alpha]$.
\end{proof}

\begin{proof}[\hypertarget{proof:lemma:good_sets_are_large}{Proof of Lemma~\ref{lemma:good_sets_are_large}}]
    By Markov's inequality and Proposition~\ref{proposition:deterministic_bound_is_achieved}, we can see that
    $$
    \bbP\left(\left[\edet{d}(Q_{d,p,\bsz})\right]^{1/\lambda} \leq \frac{2}{(1-\tau) \, p} \, \muquant\right)
    \geq
    \tau
    .
    $$
    This implies that, for each $\lambda > 1/2$, there are at least $\ceil{\tau p^d}$ such vectors which satisfy
    $$
    \edet{d}(Q_{d,p,\bsz})
    \leq
    \left(\frac{2}{(1-\tau) \, p} \, \muquant \right)^{\lambda}
    .
    $$
    The bound above approaches infinity as $\lambda$ approaches $\alpha$, see~\eqref{eq:mu_quantity},
    meaning that the minimal value of
    $$
    \left(\frac{2}{(1-\tau) \, p}\, \muquant\right)^{\lambda}
    $$
    is attained for some $\lambda$ and we can therefore say that all the vectors satisfying the inequality for this minimising $\lambda$ will satisfy it for all other $\lambda$ in the range.
\end{proof}

\begin{proof}[\hypertarget{proof:lemma:lower_bound_on_r}{Proof of Lemma~\ref{lemma:lower_bound_on_r}}]
    For prime $p$ and $\bsz \in G_\tau^{(p)}$, as defined in~\eqref{eq:good_sets}, we have
    $$
    \sum_{\substack{\bsh \in \Zdnull \\ \bsh \cdot \bsz \equiv_p 0}} \ralpha^{-2}(\bsh)
    \leq
    \left(\frac{2}{(1-\tau) \, p} \, \muquant\right)^{2 \lambda}
    \qquad \forall \lambda \in [1/2,\alpha)
    .
    $$
    This implies that whenever $\bsh \in \Z^d \setminus \{\bszero\}$ satisfies $\bsh \cdot \bsz \equiv_p 0$ we have
    \begin{equation*}
    \ralpha(\bsh)
    \ge
    p^{\lambda} \left[\frac{2}{1-\tau} \, \muquant \right]^{-\lambda}
    \qquad \forall \lambda \in [1/2,\alpha)
    .
    \end{equation*}
    The remainder of the statement follows from $p > n/2$ for $p \in \Pn$.
\end{proof}

\begin{proof}[\hypertarget{proof:lemma:omega_bound}{Proof of Lemma~\ref{lemma:omega_bound}}]
    We note that we are in the situation with $\bsh \ne \bszero$.
    Using Lemma~\ref{lemma:averaging_over_vecs} and $p > n/2$ for $p \in \Pn$, we obtain
    \begin{align*}
      \omega_{n,\tau}(\bsh)
      \le
      \frac{1}{|\Pn|} \sum_{p \in \Pn} \left[
        \frac{1}{\tau p}
        +
        \frac{\tau p - 1}{\tau p} \, \bbone(\bsh \equiv_p \bszero)
      \right]
      &<
      \frac{2}{\tau n}
      +
      \frac{1}{|\Pn|} \sum_{p \in \Pn} \bbone(\bsh \equiv_p \bszero)
      \\
      &=
      \frac{2}{\tau n}
      +
      \frac{1}{|\Pn|} \sum_{p \in \Pn} \prod_{j \in \supp(\bsh)} \bbone(h_j \equiv_p 0)
      \\
      &\le
      \frac{2}{\tau n}
      +
      \frac{1}{|\Pn|} \sum_{j \in \supp(\bsh)} \sum_{p \in \Pn} \bbone(h_j \equiv_p 0)
      ,
    \end{align*}
    where we used $\bsh \ne \bszero$ for the last line.
    We now want to obtain an upper bound on the number of $p \in \Pn$ which divide any of the nonzero $h_j$ in $\bsh$, which is the sum over $p \in \Pn$ in the last line above.
    For $h \in \N$, i.e., $h \ne 0$, consider its prime factorisation $h = \prod_{i=1}^r p_i^{m_i}$ with $p_1 < \cdots < p_r$ and all $m_i \in \N$, then, for real $q > 1$
    we can bound the number of prime divisors larger than or equal to $q$ of $h$ by $\log_q(h)$ since
    $$
      \sum_{\substack{\text{prime } p \\ p \ge q}} \bbone(h \equiv_p 0)
      =
      \sum_{\substack{i=1 \\ p_i \ge q}}^r 1
      \le
      \sum_{\substack{i=1 \\ p_i \ge q}}^r \underbrace{m_i \log_q(p_i)}_{\ge 1}
      +
      \sum_{\substack{i=1 \\ p_i < q}}^r \underbrace{m_i \log_q(p_i)}_{> 0}
      =
      \log_q(h)
      .
    $$
    Hence, using $|\Pn| > \constPn \, n / \ln n$, see~\eqref{eq:PNT_lowerbound}, and $1/\ln(n/2) \le 2/\ln n$ for $n \ge 4$, and $0 < \lambda \le \alpha$, we obtain, for $\bsh \in \Z^d \setminus \{\bszero\}$,
    \begin{align*}
      \frac{2}{\tau n}
      +
      \frac{1}{|\Pn|} \sum_{j \in \supp(\bsh)} \sum_{p \in \Pn} \bbone(h_j \equiv_p 0)
      &\le
      \frac{2}{\tau n}
      +
      \frac{1}{|\Pn|} \sum_{j \in \supp(\bsh)} \sum_{\substack{\text{prime } p \\ p > n/2}} \bbone(h_j \equiv_p 0)
      \\
      &\le
      \frac{2}{\tau n}
      +
      \frac{1}{|\Pn|} \sum_{j \in \supp(\bsh)} \frac{\ln(|h_j|)}{\ln(n/2)}
      \\
      &\le
      \frac{2}{\tau n}
      +
      \frac{2}{\constPn n} \ln\Big(\prod_{j\in\supp(\bsh)} |h_j|\Big)
      \\
      &\le
      \frac{2}{\tau n}
      +
      \frac{2}{\constPn n} \ln\Big(\prod_{j\in\supp(\bsh)} |h_j|^{\alpha/\lambda}\Big)
      \\
      &\le
      \frac{2}{\tau n}
      +
      \frac{2}{\constPn n \delta} \underbrace{\prod_{j\in\supp(\bsh)} |h_j|^{\alpha \delta/\lambda}}_{\ge 1}
      \\
      &\le
      \frac{1}{n} \left( \frac{2}\tau + \frac{2}{\constPn \delta} \right)  \gamma_{\supp(\bsh)}^{\delta/\lambda} \, \ralpha^{\delta/\lambda}(\bsh)
      ,
    \end{align*}
    where we used $\ln x \le (x^\delta-1)/\delta \le x^\delta/\delta$ for $x \ge 0$ and any $\delta > 0$, see \cite[Equation~4.5.5]{DLMF}.
    We finish the proof by using $\gamma_{\supp(\bsh)}^{\delta/\lambda} \le \max_{\emptyset \ne \setu \subset \N} \gamma_{\setu}^{\delta/\lambda} \le \max_{\emptyset \ne \setu \subset \N} \gamma_{\setu}$ for $\delta \le \lambda$.
\end{proof}

\begin{proof}[\hypertarget{proof:lemma:Gtilde_are_large}{Proof of Lemma~\ref{lemma:Gtilde_are_large}}]
    We recall the bound from Proposition~\ref{proposition:bound_on_theta},
    $$
    \frac{1}{p} \sum_{z_s^{(p)} \in \Z_p} \left[\theta_s^{(p)}(z_s^{(p)})\right]^{1/2\lambda}
    \leq
    \frac{2}{p} \sum_{\substack{\bsh\in\bbZ^{s} \\ h_s \ne 0}}
    \ralpha^{-1/\lambda}(\bsh)
    .
    $$
    We can use Markov's inequality to see that
    $$
    \mathbb{P}\bigg(\left[\theta_s^{(p)}(z_s^{(p)})\right]^{1/2\lambda} \geq \frac{2}{(1-\tau) \, p} \sum_{\substack{\bsh\in\bbZ^{s} \\ h_s \ne 0}}
    \ralpha^{-1/\lambda}(\bsh)\bigg)
    \leq
    1-\tau
    ,
    $$
    and therefore
    $$
    \mathbb{P}\bigg(\left[\theta_s^{(p)}(z_s^{(p)})\right]^{1/2\lambda} < \frac{2}{(1-\tau) \, p} \sum_{\substack{\bsh\in\bbZ^{s} \\ h_s \ne 0}}
    \ralpha^{-1/\lambda}(\bsh)\bigg)
    \geq
    \tau
    .
    $$

    This implies that at least $\ceil{\tau p}$ such vectors satisfy
    $$
    \theta_s^{(p)}(z_s^{(p)})
    \leq
    \bigg(\frac{2}{(1-\tau) \, p} \sum_{\substack{\bsh\in\bbZ^{s} \\ h_s \ne 0}}
    \ralpha^{-1/\lambda}(\bsh)\bigg)^{2\lambda}
    $$
    for each $\lambda \in [1/2, \alpha)$.
    Similarly to the proof of Lemma~\ref{lemma:good_sets_are_large}, the minimum of the right-hand side is attained for some $\lambda$ and so values of $z_s^{(p)}$ satisfying the inequality for this minimising $\lambda$ will satisfy it for all other $\lambda$ in the range.
\end{proof}

\begin{proof}[\hypertarget{proof:lemma:lower_bound_on_r_cbc}{Proof of Lemma~\ref{lemma:lower_bound_on_r_cbc}}]
    This follows the proof of Lemma~\ref{lemma:lower_bound_on_r} but with the quantity $\theta^{(p)}_{s,\alpha,\bsgamma,\bsz'}(z_s)$. \\
    For prime $p$, $\bsz' \in \bbZ^{s-1}$ and $z_s \in \widetilde{G}_{s,\tau,\bsz'}^{(p)}$, with $\widetilde{G}_{s,\tau,\bsz'}^{(p)}$ as defined in~\eqref{eq:Gtilde}, we have
    \begin{align*}
        \sum_{\substack{\bsh\in\bbZ^s \\ h_s \ne 0 \\ \bsh\cdot\bsz\equiv_p 0}} \ralpha^{-2}(\bsh)
        \le
        \inf_{\lambda \in [1/2,\alpha)}
        \bigg( \frac{2}{(1-\tau) \, p} \sum_{\substack{\bsh \in\bbZ^s \\ h_s \ne 0}}
        \ralpha^{-1/\lambda}(\bsh) \bigg)^{2\lambda}
        .
    \end{align*}
    This implies that whenever $\bsh \in \Z^s \setminus \{\bszero\}$ satisfies $\bsh \cdot \bsz \equiv_p 0$ we have
    $$
    \ralpha(\bsh)
    >
    \sup_{\lambda \in [1/2,\alpha)} p^{\lambda} \bigg[\frac{2}{1-\tau} \, \sum_{\substack{\bsh \in\bbZ^s \\ h_s\ne 0}} \ralpha^{-1/\lambda}(\bsh) \bigg]^{-\lambda}
    .
    $$
    The remainder of the statement follows from $p > n/2$ for $p \in \Pn$.
\end{proof}

\bibliography{Bibliography}
\bibliographystyle{abbrv}

\end{document}